\documentclass[10pt]{article}

\title{Cycle type in Hall-Paige: A proof of the Friedlander-Gordon-Tannenbaum conjecture}
\author{Alp M\"uyesser}

\usepackage{graphicx, caption}
\usepackage{amssymb}
\usepackage{amsthm}
\usepackage{amsmath}
\usepackage{enumitem}
\usepackage{framed}
\usepackage{soul}
\usepackage{changepage}
\usepackage{comment}
\usepackage{float}
\usepackage{hyperref}
\usepackage[dvipsnames, table]{xcolor}
\usepackage{lscape}

\usepackage{mathtools}

\theoremstyle{plain}
\newtheorem{theorem}{Theorem}[section]
\newtheorem{lemma}[theorem]{Lemma}

\newtheorem{proposition}[theorem]{Proposition}
\newtheorem{claim}{Claim}[theorem]
\newtheorem{conjecture}[theorem]{Conjecture}
\newtheorem{observation}[theorem]{Observation}
\newtheorem{corollary}[theorem]{Corollary}
\theoremstyle{definition}
\newtheorem{definition}[theorem]{Definition}

\newtheorem{remark}[theorem]{Remark}

\newcommand\eps{\varepsilon}

\newcommand{\cH}{\mathcal{H}}

\newcommand{\cM}{\mathcal{M}}
\newcommand{\bN}{\mathbb{N}}

\newcommand{\cW}{\mathcal{W}}

\newcommand{\doublesquig}{%
  \mathrel{%
    \vcenter{\offinterlineskip
      \ialign{##\cr$\rightsquigarrow$\cr\noalign{\kern-1.5pt}$\rightsquigarrow$\cr}%
    }%
  }%
}

\usepackage{graphicx}%

\newcommand{\id}{{e}}

\begin{document}
\maketitle

\begin{abstract}
    An orthomorphism of a finite group $G$ is a bijection $\phi\colon G\to G$ such that $g\mapsto g^{-1}\phi(g)$ is also a bijection.  In 1981, Friedlander, Gordon, and Tannenbaum conjectured that when $G$ is abelian, for any $k\geq 2$ dividing $|G|-1$, there exists an orthomorphism of $G$ fixing the identity and permuting the remaining elements as products of disjoint $k$-cycles. We prove this conjecture for all sufficiently large groups.
\end{abstract}

\section{Introduction}

An \textbf{orthomorphism} of a finite group $G$ is a bijection $\phi\colon G\to G$ such that $g\mapsto g^{-1}\phi(g)$ is also bijective. Orthomorphisms have attracted much interest in recent years, not least due to their link with Latin squares. The multiplication tables of groups with orthomorphisms yield Latin squares with orthogonal mates, which in turn give useful constructions in design theory (see the book of Evans \cite{evans2018orthogonal} for an overview of the area). A fundamental conjecture in the area is the Hall-Paige conjecture \cite{hallpaige} which states that a group $G$ admits an orthomorphism if and only if the product of all elements in the group (in any order) belongs to the commutator subgroup (this property is henceforth referred to as the \textbf{Hall-Paige condition}). For abelian groups, the Hall-Paige condition simply means that the sum of all elements in the group is the identity. The conjecture was confirmed by Wilcox~\cite{wilcox1}, Evans~\cite{evans}, and Bray~\cite{BRAY} in 2009. It is not too difficult to see that the Hall-Paige condition is necessary, but the fact that it is also sufficient is quite remarkable.
\par The proof of Wilcox, Evans and Bray has the disadvantage that it relies extensively on the classification of finite simple groups. Recently, two new proofs of the Hall-Paige conjecture have been found which do not rely on this classification, with the caveat that both proofs require the group to be sufficiently large. On the other hand, both proofs strengthen the original statement of the Hall-Paige conjecture in a distinct, novel direction. The first of these proofs is due to Eberhard, Manners, and Mrazovi\'c \cite{asymptotichallpaige}. This proof uses tools from analytic number theory, and it yields a strikingly accurate asymptotic on the number of orthomorphisms for groups with the Hall-Paige condition. The second proof is due to the author and Pokrovskiy \cite{muyesser2022random}, and this proof has the advantage of finding orthomorphisms in random-like subsets of groups. This flexibility turns out to be quite fruitful as demonstrated by the numerous applications of the ``random Hall-Paige conjecture'' given in \cite{muyesser2022random}.
\par The current paper is focused a third way to strengthen the Hall-Paige conjecture, this time by asserting the existence of orthomorphisms with specific cycle types. Recall that the \textbf{cycle type} of a permutation $\pi$ encodes how many cycles of each length are present when $\pi$ is written as a product of disjoint cycles. For example, orthomorphisms that consist of a single cycle come up naturally in Ringel's resolution of the Heawood map colouring conjecture, which motivated Ringel to ask for a classification of all groups with such orthomorphisms (see \cite{ringel2012map, friedlander1978group, alspach2017friedlander, ringeloldproblem}). Several other problems of a similar flavour concerning ``sequenceable groups'' were raised by numerous authors with the motivation to construct Latin squares with additional properties (see \cite{ollis2002sequenceable} and Section 1.1.2 in \cite{muyesser2022random}). There are also motivations to study orthomorphisms with other cycle types. For example, orthomorphisms that are products of disjoint $6$-cycles give constructions such as ``cyclic'' Steiner triple systems \cite{johnsen1974combinatorial}. 
\par A unifying conjecture in the area was given by Friedlander, Gordon, and Tannenbaum in 1981 \cite{friedlander}.
\begin{conjecture}[The Friedlander-Gordon-Tannenbaum (FGT) conjecture, 1981]\label{conj:FGT}
Let $G$ be an abelian group of order $n$ satisfying the Hall-Paige condition. Suppose for some integer $k\geq 2$ that $k$ divides $n-1$. Then, there exists an orthomorphism of $G$ that fixes the identity element, and permutes the remaining elements as products of disjoint cycles of length $k$.
\end{conjecture}
\par The Hall-Paige conjecture is not very laborious to verify for abelian groups, and this was already achieved by Hall and Paige when they posed their conjecture. The FGT conjecture, on the other hand, has remained open for more than forty years. There are several partial results towards the FGT conjecture in the literature. Friedlander, Gordon, and Tannenbaum themselves confirmed their conjecture for groups of order at most $15$, and abelian $p$-groups where $p\geq 3$ \cite{friedlander}. We refer the reader to \cite{evans2018orthogonal} for a more detailed overview (see also \cite{bors2022coset, bors2022cycle, wang1994harmoniousness} for results about the very related concept of complete mappings). We just remark that the $k=3$ and the cyclic group case of the FGT conjecture is open, signifying the difficulty of the problem. In this paper, we resolve the FGT conjecture for sufficiently large groups.

\begin{theorem}\label{thm:mainthmintro}
The Friedlander-Gordon-Tannenbaum conjecture is true for all sufficiently large groups.
\end{theorem}
We use methods from probabilistic combinatorics, so our proof needs large groups just to get concentration for some random variables with fairly simple distributions. We do not make this constant explicit to make the presentation neater. We make three further remarks.
\begin{remark}
At the time the FGT conjecture was posed, the Hall-Paige conjecture was known to be true for abelian groups, but not in general, which perhaps explains why Conjecture~\ref{conj:FGT} is concerned only with abelian groups. Given the present work, it seems reasonable to suspect that the FGT conjecture can be extended to non-abelian groups, perhaps even quasi-groups/Latin squares, which would generalise the famous Ryser-Brualdi-Stein conjecture. We discuss this further in the concluding remarks, Section~\ref{sec:concluding}.
\end{remark}
\begin{remark}
Our proof of Theorem~\ref{thm:mainthmintro} actually gives much more, and can be used to give many other cycle types that can be realised via orthomorphisms. We discuss this further in Section~\ref{sec:concluding}.
\end{remark}
\begin{remark}\label{completemapping}
    A very related notion is that of a \textbf{complete mapping}, which is a permutation $\phi$ of a group $G$ such that $g\to g\phi(g)$ is also bijective. A group admits a complete mapping if and only if it admits an orthomorphism, essentially because the map $g\to g^{-1}$ is a bijection. Therefore, the Hall-Paige conjecture is sometimes stated with respect to complete mappings instead of orthomorphisms. However, this equivalence does not hold when we make restrictions on the cycle type. For example, in an abelian group, there cannot be a complete mapping inducing any cycle of length $2$, therefore the FGT conjecture does not hold when orthomorphisms are replaced with complete mappings (for a more detailed discussion of cycle types of complete mappings, see \cite{bors2022coset, bors2022cycle}). However, some appropriate modification of the FGT conjecture likely holds for complete mappings as well, and we discuss this further in Section~\ref{sec:concluding}. We should also remark that, confusingly, orthomorphisms are called complete mappings in \cite{friedlander}, but the convention in the current paper seems to be standard following the book of Evans \cite{evans2018orthogonal}.
\end{remark}
The next section serves as a skeleton for the paper. In this section, we give a bird's eye view of the proof of Theorem~\ref{thm:mainthmintro}, and break up the task into proving two key lemmas.

\section{Main theorem and overview of the proof}\label{sec:overview}
\subsection{Definitions of key auxiliary graphs and hypergraphs}
It is customary in combinatorics to rephrase statements such as Conjecture~\ref{conj:FGT} in terms of finding perfect matchings in hypergraphs, or finding rainbow structures in edge-coloured graphs, and we follow this tradition in the current paper.
\par Given a group of order $n$, we denote by $\vec{K}_G$ the edge-coloured directed graph defined as follows. $V(\vec{K}_G):=G$, and $E(\vec{K}_G):=\{(a,b)\in G\times G\colon a\neq b\}$, and the \textit{colour} of an edge $(a,b)$ is the group element $ab^{-1}$. Given subsets $V,C\subseteq G$, by $\vec{K}_G[V;C]$ we denote the subgraph of $\vec{K}_G$ obtained by keeping only the vertices in $V$, and the directed edges with colours in $C$. Occasionally, the following related definition will also be useful. Given multiple subsets $V_1,V_2, \ldots, V_k\subseteq G$, we denote by $\vec{K}_G[V_1,V_2, \ldots, V_k]$ the edge-coloured directed graph with vertex set $V_1\sqcup V_2 \sqcup \cdots \sqcup V_k$ ($\sqcup$ indicates that we are taking a disjoint union) and edge set consisting of edges of the form $e=(v,w)\in V_i\times V_{i+1}$ (with colour $vw^{-1}$) for some $i\in \{1,2,\ldots, k\}$ (where $k+1=1$). By $\vec{K}_G[V_1,V_2, \ldots, V_k; C]$, we denote the same graph obtained by keeping only edges whose colour is in $C$.
\par Recall that a subgraph of an edge-coloured graph is called \textbf{rainbow} if all edges have distinct colours. Given $V,C\subseteq G$, let $\mathcal{H}_k[V;C]$ be the $2k$-uniform hypergraph on the vertex set $V\sqcup C$ where $v\sqcup c$ is an edge whenever $v\subseteq V$ induces a rainbow directed cycle of length $k$ in $\vec{K}_G$ with the colour set of the cycle being precisely $c$. $\mathcal{H}_k$ denotes $\mathcal{H}_k[G;G]$. Sometimes we overload the terms vertex and colour by referring to elements of $V(\mathcal{H}_k[V;C])$ which come from $V$ as \textbf{vertices} and those which come from $C$ as \textbf{colours}. The following observation is quite critical.
\begin{observation}\label{obs:key}
    If $c$ is the colour set of an edge in $\mathcal{H}_k$, or the colour set of some directed rainbow cycle in $\vec{K}_G$ (of any length), then the sum of all the elements of $c$ must equal $0$, i.e. $c$ is a zero-sum set.
\end{observation}
\begin{proof}
    As in a directed cycle each vertex has one in-edge and one out-edge, when we take a sum of all the colours of a cycle in $\vec{K}_G$, each vertex appears twice, once positive, and once negative. The statement follows.
\end{proof}
\par Given graphs $H$ and $G$, we say that $G$ contains an $H$-\textbf{factor} if there exists a collection of copies of $H$ in $G$ that partition the vertex set of $G$. For example, a $K_2$-factor in a graph is a perfect matching. $\vec{P}_k$ denotes a directed path of length $k$ (meaning with $k$ edges). $\vec{C}_k$ denotes a directed cycle of length $k$ (meaning with $k$ vertices and $k$ edges). The following proposition follows from all the definitions presented thus far.

\begin{proposition}\label{prop:reduction}
    Let $G$ be a finite abelian group and let $k$ be an integer with $k\geq 2$. The following are equivalent.
    \begin{itemize}
        \item $G$ admits an orthomorphism fixing the identity and permuting the remaining elements as products of disjoint $k$-cycles.
        \item $\vec{K}_G[G\setminus\{0\}; G\setminus\{0\}]$ contains a rainbow $\vec{C}_k$-factor.
        \item $\mathcal{H}_k[G\setminus\{0\}; G\setminus\{0\}]$ has a perfect matching.
    \end{itemize}
\end{proposition}
We invite the reader to verify the above proposition.
Thanks to Proposition~\ref{prop:reduction}, we can phrase our main result in the language of hypergraph matchings in the next subsection.

\subsection{Main theorem and its proof modulo key lemmas}
Recall that a $p$-\textbf{random} subset of set $S$ is one obtained by sampling each element of $S$ independently with probability $p$. Similarly, we say a collection of random sets $R_1,\ldots, R_k\subseteq S$ is \textbf{disjoint} $p$-\textbf{random} if each element of $S$ belongs to each $R_i$ with probability $p$, and to none of the $R_i$ with probability $1-pk$, and these decisions are made independently for each element of $S$. We reserve the letter $n$ for the size of the ambient group throughout the paper. When we say that an event holds ``with high probability'', we mean that the probability of the event approaches $1$ as $n$ tends to infinity.
\begin{theorem}[Main theorem]\label{thm:mainthm}
There exists an absolute constant $\eps_{\ref{thm:mainthm}}>0$ such that the following holds. Let $G$ be an abelian group of order $n$, let $p\geq n^{-\eps_{\ref{thm:mainthm}}}$, and suppose $k$ is some integer such that $3\leq k\leq \log ^{10} n$. Let $R_1,R_2\subseteq G$ be $p$-random subsets, sampled independently. Then, the following holds with high probability.
\par Let $V,C\subseteq G$ be equal-sized subsets with $|V\Delta R_1|,|C\Delta R_2|\leq n^{3/4}$. Suppose $k$ divides $|V|$ (and thus, $|C|$), and suppose $\sum C=0$. Then, $\mathcal{H}_k[V; C]$ has a perfect matching.
\end{theorem}
Theorem~\ref{thm:mainthm} turns into a deterministic statement when applied with $p=1$. This statement, when $n$ is sufficiently large, implies Conjecture~\ref{conj:FGT} (when $3\leq k\leq \log ^{10} n$, the ``low-girth case'') by setting $V=C=G\setminus \{0\}$. Theorem~\ref{thm:mainthm} can thus be interpreted as a randomised version of Conjecture~\ref{conj:FGT}. As far as our proof method is concerned, it does not take extra work to prove Theorem~\ref{thm:mainthm} compared to Conjecture~\ref{conj:FGT}. Theorem~\ref{thm:mainthm} also has further applications. Using its full strength, one can find orthomorphisms with other cycle types, see Section~\ref{sec:concluding} for more details.
\par The $k=2$ case of Conjecture~\ref{conj:FGT} is proven implicitly in \cite{friedlander}, where the authors give orthomorphisms of odd-order cyclic groups which are products of disjoint transpositions (see also \cite{evans2018orthogonal} for a proof of the $k=2$ case). The case of $k>\log ^{10} n$ (the ``high-girth case''), on the other hand, can be resolved by using some tools from \cite{muyesser2022random}. In fact, the $k=n$ case (the Hamilton cycle case) was implicitly solved in \cite{muyesser2022random} already, and it turns out the method is general enough to handle cycles of length at least polylogarithmic in $n$. We give the details for this in Section~\ref{sec:highgirth}. We remark that for the methods of \cite{muyesser2022random}, this polylogarithmic lower bound on the cycle length is a hard barrier, essentially because any sorting network (see \cite{batchersorting, ajtaisorting}) must have depth at least $\log n$.
\par For the rest of this section, we focus on Theorem~\ref{thm:mainthm}, which is concerned with the ``low-girth'' case of the FGT conjecture. The key lemma used to prove Theorem~\ref{thm:mainthm} is the following, which states the existence of an ``absorber for zero-sum subsets''. Roughly speaking, this lemma states that random subsets contain ``absorbers'' which have the ability to combine with any small enough set to produce matchings (we say that the small set is ``absorbed''), provided that this small set satisfies some straightforward necessary conditions.
\begin{lemma}[Zero-sum absorption]\label{lem:zerosumabsorption}
There exist absolute constants $\eps=\eps_{\ref{lem:zerosumabsorption}}>0$ and $K=K_{\ref{lem:zerosumabsorption}}\geq 1$ with $\eps K\leq 10^{-10}$ such that the following holds. Let $3\leq k\leq \log^{10} n$, $p\geq n^{-\eps}$. Let $R_1, R_2\subseteq G$ be $p$-random subsets, sampled independently. Let $m\in k\cdot \mathbb{N}$ with $m \leq (p/k \log n)^K n$. Then, the following holds with high probability.
\par Let $U\subseteq G$ with $|U|\leq n^{4/5}$. Then, there exist $V\subseteq R_1\setminus U$ and $C\subseteq R_2\setminus U$ with the following property. For any $V'\subseteq G\setminus V$ and $C'\subseteq G\setminus C$ with $|V'|=|C'|=m$, $\sum C' = 0$, $0\notin C'$, we have that $\vec{K}_G[V\cup V'; C\cup C']$ has a rainbow $\vec{C}_k$-factor, or equivalently, $\mathcal{H}_k[V\cup V'; C\cup C']$ has a perfect matching.
\end{lemma}
A lemma of this form, guaranteeing \textit{some} rainbow directed $2$-factor (a partition of the vertex set into cycles) was already proved in \cite{muyesser2022random}, and we refer the reader there for more context regarding the absorption method. Our goal here is to achieve precise control over the cycle sizes of the $2$-factor we find, which corresponds to controlling the cycle type of the underlying orthomorphism. This is a delicate task requiring several novel ideas, and we discuss the proof of Lemma~\ref{lem:zerosumabsorption} in Section~\ref{sec:overabsorb}. We now discuss how Lemma~\ref{lem:zerosumabsorption} allows us to prove Theorem~\ref{thm:mainthm}.
\par Absorbers are helpful because after they are found and removed from the hypergraph, one can usually find a nearly spanning matching $M_1$ covering all but a vanishing proportion of the vertices. The few uncovered vertices then combine with the absorber to produce another matching, $M_2$, which means that $M_1\cup M_2$ is a perfect matching in the original hypergraph. To find a matching like $M_1$, a typical tool to use would be the R\"odl nibble \cite{AS}, which by this point refers to a wide variety of results of the following form: Let $\mathcal{H}$ be nearly regular, have small maximum co-degree\footnote{Co-degree of a pair of vertices in a hypergraph is the number of edges which contain both vertices.} and small uniformity, then $\mathcal{H}$ contains a matching covering all but a vanishing proportion of the vertices. $\mathcal{H}_k$ has the first two of these properties but not necessarily the last property when $k\gg \log n$, but let us put this technicality aside for now. An equally important technicality is that greedily removing an absorber from $\mathcal{H}_k$ might damage the regularity of $\mathcal{H}_k$ too much for the R\"odl nibble to work. To remedy this, we find the absorber inside smaller random subsets, so that when the absorber is removed, it is only a small section of the hypergraph whose regularity gets spoiled. Indeed, the complement of the smaller random set is still random, and therefore inherits almost regularity by standard concentration arguments. To finish, we need a version of R\"odl nibble that works with regular hypergraphs plus a few ``junk vertices'', which correspond to the leftover in the smaller random sets after the absorber is removed. The following lemma encapsulates this idea.

\begin{lemma}\label{lem:deterministicnibble} There exists an absolute constant $\eps_{\ref{lem:deterministicnibble}}>0$ such that the following holds. Let $G$ be an abelian group of order $n$. Let $3 \leq k\leq \log^{10} n $, and let $p\geq n^{-\eps_{\ref{lem:deterministicnibble}}}$. Let $R_1$ and $R_2$ be $p$-random subsets of $G$, sampled independently. The following holds with high probability.
\par For any $|V_D|=|C_D|\subseteq G$ with $|V_D|,|V_C|\leq \eps_{\ref{lem:deterministicnibble}} p^3n/k^{100}$, $\mathcal{H}_k[R_1\cup V_D; R_2\cup C_D]$ contains a matching covering all but at most $n^{1-1/10^8}$ vertices.
\end{lemma}
\par The proof of Lemma~\ref{lem:deterministicnibble} comes down to establishing certain pseudorandomness properties of $\mathcal{H}_k$. Checking pseudorandomness in hypergraphs is notoriously tricky, for example see \cite{haviland1992testing} for a useful criterion for dense hypergraphs. Unfortunately, $\mathcal{H}_k$ is quite sparse, and potentially has large uniformity, so \cite{haviland1992testing} is not immediately useful in our set-up. For this reason, we have to put a fair bit of care into the proof of Lemma~\ref{lem:deterministicnibble}.
\par We can now give the proof of our main theorem, assuming these two lemmas. We remark that often in our proofs, we have random subsets $R'\subseteq R$ where $R$ itself is a random subset of the group $G$. When we say that $R'$ is a $q$-random subset, we always mean that $R'$ is a $q$-random subset of the group $G$, and not of $R$.
\begin{proof}[Proof of Theorem~\ref{thm:mainthm}] Pick a value of $\eps_{\ref{thm:mainthm}}$ such that $0<\eps_{\ref{thm:mainthm}}\ll \eps_{\ref{lem:zerosumabsorption}},\eps_{\ref{lem:deterministicnibble}}, 1/K_{\ref{lem:zerosumabsorption}}$.
For each $i\in \{1,2\}$, partition $R_i$ into $R_i^{(1)}$ and $R_i^{(2)}$ which are disjoint $p_1$-random and $p_2$-random sets respectively, where $p_1=(1/10)\eps_{\ref{lem:deterministicnibble}}p^4/k^{100}$ (and $p_2=p-p_1$). We have that $p_1\geq n^{-\eps_{\ref{lem:zerosumabsorption}}}$ and $p_2\geq n^{\eps_{\ref{lem:deterministicnibble}}}$ if $\eps_{\ref{thm:mainthm}}$ is small enough. Select some $m\in k\cdot \mathbb{N}$ such that $10n^{1-1/10^8}\leq m\leq (p_1/k\log n)^K n$ (there exists such values of $m$ as $\eps_{\ref{lem:zerosumabsorption}}K_{\ref{lem:zerosumabsorption}}\leq 10^{-10}$). With high probability, Lemma~\ref{lem:zerosumabsorption} holds with $(R_1^{(1)},R_2^{(1)})$ with this value of $m$ and Lemma~\ref{lem:deterministicnibble} holds with $(R_1^{(2)},R_2^{(2)})$. Also with high probability, the size of each random set is at most $\sqrt{n}\log n$ away from its expectation (by Chernoff's bound, see Lemma~\ref{chernoff}). With high probability, all of these properties hold simultaneously.
\par Now, fix random sets having all these properties and let $V$ and $C$ be given as in the statement of the theorem. Set $U:=(R_1\setminus V)\cup (R_2\setminus C)$ noting $|U|\leq 2n^{3/4}$.  Apply Lemma~\ref{lem:zerosumabsorption} to find absorbing subsets $V_A\subseteq R_1^{(1)}\setminus U\subseteq V$ and $C_A\subseteq R_2^{(1)}\setminus U\subseteq C$ which can combine with $m$-sized vertex-sets and $m$-sized zero-sum colour-sets to produce perfect matchings. Note this implies in particular that $|V_A|=|C_A|$. Set $V_D:=(R_1^{(1)}\cap V)\setminus V_A$ and $C_D:=(R_2^{(1)}\cap C)\setminus C_A$, noting $|V_D|,|C_D|\leq 2p_1n \leq \eps_{\ref{lem:deterministicnibble}}p_2^3n/k^{100}$. Note also that $||V_D|-|C_D||\leq 10n^{3/4}$. By Lemma~\ref{lem:deterministicnibble}, $\mathcal{H}_k[V_D\cup (R_1^{(2)}\cap V); C_D\cup (R_2^{(2)}\cap C)]$ has a matching $M_1$ covering all but at most $10n^{1-1/10^8}$ vertices (formally, we delete $\leq 10n^{3/4}$ elements from $C_D$ or $V_D$ so that $|V_D|=|C_D|$, also we initially find the matching inside $\mathcal{H}_k[V_D\cup R_1^{(2)}; C_D\cup R_2^{(2)}]$ and delete the $\leq |U|$ matched edges that use a vertex/colour from $U$). If necessary, unmatch some edges of $M_1$ so that the number of leftover vertices $V':=V\setminus V(M_1)$ and colours $C':=C\setminus V(M_1)$ are both equal to $m$ (possible as $k$ divides $|V|,|C|$ and $m$ and $|V_A|=|C_A|$). Note that the matching $M_1$ guarantees that all colours in $C\setminus C'$ admit a partition $C_1,C_2,\ldots$ where each $C_i$ is the colour set of a rainbow cycle in $\vec{K}_G$, meaning that $\sum C_i=0$ for each $i$ (see Observation~\ref{obs:key}). So we must have $\sum C\setminus C'=0$ also. As $\sum C=0$ by assumption, this implies that $\sum C'=0$, so we can invoke the property coming from Lemma~\ref{lem:zerosumabsorption}. This means that $V'$ and $C'$ combine with $V_A$ and $C_A$ to produce a matching, $M_2$. $M_1\cup M_2$ is then the desired perfect matching.
\end{proof}

\subsection{Overview of proof of the absorption lemma}\label{sec:overabsorb}
 We now discuss the proof of Lemma~\ref{lem:zerosumabsorption}, the key technical result in the paper. The general framework for constructing absorbers is (1) finding ``many'' small (meaning typically $O(1)$-sized) ``gadget'' subgraphs which give ``local variability'', and (2) aligning various gadgets effectively to obtain a large structure with ``global variability'' such as $V,C$ in Lemma~\ref{lem:zerosumabsorption}. We now expand on what such a strategy entails in the context of $\mathcal{H}_k$.
 \par To give an example of a ``gadget with local variability'', let us focus on the case where $k=3$ and $G=\mathbb{Z}_n$, the cyclic group with $n$ elements. Consider distinct $a,b,c,d\in \mathbb{Z}_n$ such that $(a,b,c)$ and $(b,c,d)$ both induce a rainbow $\vec{C}_3$ in $\vec{K}_{\mathbb{Z}_n}$ and suppose these two rainbow triangles use the same exact colour set which we'll call $C$ (this means that $a-b=c-d$ and $c-a=d-b$, and also that $C=\{a-b,b-c,c-a\}=\{b-c,c-d,d-b\}$). Then, we can say that $\{b,c\} \sqcup C$ is a gadget that can locally vary between $a$ and $d$ in the sense that $\{a\}\cup \{b,c\}\sqcup C$ and $\{d\}\cup \{b,c\}\sqcup C$ both induce matchings in $\mathcal{H}_k$. One can interpret being able to combine with both $a$ and $d$ to create perfect matchings as a severe weakening of the property demanded by Lemma~\ref{lem:zerosumabsorption}.
\par What we mean when we say there are ``many'' gadgets of this sort is the following. For any fixed $a$ and $d$ and small subset $S\subseteq \mathbb{Z}_n$ (think of size $o(n)$), we want that there exists a choice for $b, c\in \mathbb{Z}_n\setminus S$ and $C\subseteq \mathbb{Z}_n\setminus S$ so that $\{b,c\}\sqcup C$ can locally vary between $a$ and $d$. So we want to be able to find gadgets that vary between fixed vertices even after a small portion of the group is deleted by an adversary. This requirement can be rephrased as saying we want our gadgets to be ``well-distributed'' within our host structure. Such a property is very helpful as it plays nicely with the existence of robustly matchable bipartite graphs, which is a highly versatile tool devised by Montgomery to construct absorbers \cite{randomspanningtree}. This approach is sometimes called ``distributive absorption'', and was also used to settle Ringel's tree packing conjecture, see \cite{ringel} for more details.
\par One can prove by hand that for any $a,d$, the gadgets $\{b,c\}\sqcup C$ satisfying the desired properties in $\mathbb{Z}_n$ are well-distributed in the above sense. Indeed, letting $b$ range over $\mathbb{Z}_n$, and setting $c=d+a-b$ will have the desired outcome. However, this strategy would fail in $(\mathbb{Z}_2)^k$, the elementary abelian $2$-group, for the following reason. With this strategy, the colour of the edge $(b,c)$ is always $b-(d+a-b)=2b-d-a=a+d$. We think of $a$ and $d$ as constants, so in this sense letting $b$ range over $(\mathbb{Z}_2)^k$ does not change the colour of the edge $(b,c)$. This means that if the colour $a+d$ is deleted by an adversary, we cannot find a gadget switching between $a$ and $d$ in $(\mathbb{Z}_2)^k$.
\par This exemplifies just one type of technicality we need to watch out for while designing gadgets. What we ask for from our gadgets can be encoded by systems of linear equations, and we want there to be a a well-distributed set of solutions to this system. On top of that, we want the solutions to have \textit{distinct coordinates}. Indeed, if $b=c$, the gadget $\{b,c\}\sqcup C$ would be useless, as $(a,b,c)$ wouldn't correspond to a $\vec{C}_3$ in $\vec{K}_G$. In Section~\ref{freeproducts}, we introduce a set of sufficient conditions for a system of equations to correspond to a well-distributed collection of gadgets. The tricky part of the task is finding gadgets that are powerful enough to display some local variability, but also weak enough to satisfy these sufficient conditions. We refer the reader to Figure~\ref{fig:gadgets} which contains the building blocks of most of the gadgets we use in the paper.
\par Given the ease with which we found a somewhat useful gadget, perhaps the reader might be tempted to guess that finding gadgets in problems of this type should not be very arduous, at least in cyclic groups. We now make a small aside providing evidence towards the contrary, which we hope will make it somewhat surprising that the FGT conjecture is true, already for $k=3$ and cyclic groups.
\par Firstly, what we mean by ``problems of this type'' are hypergraph matching problems where vertices represent group elements and edges represent solutions to linear equations. For example, in the $k=3$ case of the FGT conjecture, $\mathcal{H}_3$ is a $6$-uniform hypergraph where the edges $(a,b,c,d,e,f)$ have to satisfy $d=b-a$, $e=c-b$, and $f=a-c$ (so $a,b,c$ represent vertices and $d,e,f$ represent colours). This is arguably more complex than the corresponding hypergraph for the Hall-Paige conjecture, where the edges $(a,b,c)$ would be $3$-uniform and satisfy $b-a=c$.
\par Another critical example comes from the toroidal version of the $n$-queens problem for which Bowtell and Keevash \cite{bowtell2021n} recently made a break-through. We won't give much context for this problem here, it suffices to know that the problem is about finding perfect matchings in hypergraphs where edges are of the form $(a,b,c,d)$ where $c=a+b$ and $d=a-b$ and the vertices come from cyclic groups\footnote{It turns out that such a perfect matching corresponds to placement of mutually non-attacking queens on a toroidal chessboard.}. It's not easy to say a priori whether this problem should be easier or harder than finding perfect matchings in $\mathcal{H}_3$. Just like for the FGT conjecture, the Hall-Paige condition is a necessary condition for the $n$-queens hypergraph to have a perfect matching. As the underlying group is cyclic, this means that the group must be of odd-order.
\par However, in stark contrast with the FGT conjecture and the Hall-Paige conjecture, the $n$-queens problem has another necessary condition which comes from squaring the group elements (see \cite{bowtell2021n} for details, and note that squaring refers to multiplication modulo $n$, as opposed to addition modulo $n$), which makes it impossible for the cyclic group to be of order $3$ modulo $6$ if it is to contain a perfect matching. Any gadget-based strategy as described here, without serious modifications, would not be able to detect this obstruction that comes from squaring group elements. Perhaps, this explains why gadgets displaying any sort of useful local variability are very difficult to find in the $n$-queens problem (in \cite{bowtell2021n}, Bowtell and Keevash use a more involved strategy combining randomised algebraic constructions and iterative absorption).
\par Given the similarity between the constraints in the $n$-queens problem and the $k=3$ case of the FGT conjecture, it is surprising that the FGT conjecture does not have any more necessary conditions in addition to the Hall-Paige condition. It is hard to give a satisfying explanation for why this is the case, and there seems to be interesting open problems in this area which we discuss in Section~\ref{sec:concluding}. To conclude this digression, we just reiterate that the existence of gadgets in hypergraph matching problems is quite delicate and sensitive to slight changes in the constraints that define the edge sets.
\par Returning to our overview of the proof of Lemma~\ref{lem:zerosumabsorption}, we've already discussed the rudiments of how to find gadgets, and that we combine the gadgets using the distributive absorption technique. These techniques are sufficient to prove a version of Lemma~\ref{lem:zerosumabsorption} where the subsets to be absorbed ($V'$ and $C'$) are contained in another subset $A$ which we get to choose, as long as it is sufficiently small (see Lemma~\ref{lemma:flexiblevertexcolourabsorber}). To prove Lemma~\ref{lem:zerosumabsorption}, we would in fact want $A$ to be the whole group, and not just a subset. To go from this weaker statement to Lemma~\ref{lem:zerosumabsorption}, we apply the former with $A$ set to be a random set, and we show that for any $V',C'$ (with $C'$ being zero-sum) we can find a matching in $A\cup V' \cup C'$ saturating $V' \cup C'$ (see Section~\ref{sec:saturating} for various lemmas of this sort). By the property of $A$, the unused vertices in $A$ which are not a part of this matching can be absorbed. This way, we reduce the task of absorbing arbitrary subsets of $\mathcal{H}_k$ to absorbing subsets that live within a small set that we get to choose. This step is analogous to the ``cover-down'' step in applications of the ``iterative absorption'' method (see for example \cite{barber2020minimalist}), though we will not require any iterative strategy here.
\par There are very important challenges that arise in the implementation of this cover-down step, especially for large $k$ (meaning $k$ tends to infinity as $n$ tends to infinity). For example, to be able to saturate colours, the following question is very relevant. For which $k$-subsets $C\subseteq G$ does $\mathcal{H}_k$ have an edge with $C$ as its colour set? Equivalently, which colour sets induce directed rainbow $k$-cycles in $\vec{K}_G$? It is not hard to see that we need that $\sum C = 0$ (see Observation~\ref{obs:key}), but thankfully, this is not so difficult to ensure. Using some techniques from \cite{muyesser2022random}, we can without loss of generality (roughly speaking) assume that $C'$ (from the statement of Lemma~\ref{lem:zerosumabsorption}) comes with a partition into $k$-sets which are all zero-sum. Note this is feasible because $C'$ is itself zero-sum by assumption.
\par Further thought reveals that to find a directed rainbow $k$-cycle in $\vec{K}_G$ with colour set $C$, we need to be able to order $C$ as $(c_1,\ldots ,c_k)$ such that $c_1$, $c_1+c_2$, $c_1+\cdots c_k$ are all distinct. This property comes from the necessity that each of the $k$ vertices of the rainbow cycle need to be distinct. Does such an ordering always exist\footnote{We would also need that $0\notin C$, but this is built into the statement of the FGT conjecture, so we ignore this technicality here.}? It turns out that even for cyclic groups of prime order, this is an open problem, posed initially by Ronald Graham in 1971. Surprisingly little is known about this problem (see Problem 10 from \cite{graham1971sums}, see also \cite{alspach2020strongly, costa2020some, costa2022sequences, costa2023sequencings}). For example, the problem is already open for $k=13$ and cyclic groups of prime order.
\par The large $k$ case is further complicated by an entirely different problem that is probabilistic in nature. Even if we start with a zero-sum $k$-subset that admits a permutation making its partial sums distinct, there still remains the problem of finding many cycles with this colour set inside a \textit{random} vertex-subset of $\mathcal{H}_k$. When $k=o(\log n)$, standard concentration tools such as Chernoff's bound resolve this issue, but for larger $k$, the situation is much less clear. Finding a cover-down strategy that avoids both of these issues for large $k$ together with a compatible distributive absorption strategy is arguably the most difficult aspect of our proof.
\par That said, neither of these issues come up when $k=3$, and in this case we obtain a more transparent proof. For this reason, we recommend the reader interested in inspecting the proof to start with assuming $k=3$ throughout on an initial read-through, ignoring any lemmas with hypotheses such as ``let $k\geq 10$''.
\subsection{Organisation of the rest of the paper} We collect some preliminary tools in
Section~\ref{sec:preliminaries}.
We have already broken up the task of proving Theorem~\ref{thm:mainthm} into proving Lemma~\ref{lem:deterministicnibble} and Lemma~\ref{lem:zerosumabsorption}. The former is done in Section~\ref{sec:nibble} and the latter is done in Section~\ref{sec:zerosumabsorption}. In Section~\ref{sec:highgirth}, we show how the high-girth case of the FGT conjecture can be derived from results from \cite{muyesser2022random}, as promised earlier on in this section. In Section~\ref{sec:concluding}, we discuss some directions for future research.
\section{Preliminaries}\label{sec:preliminaries}
\subsection{Probabilistic tools}
\subsubsection{Concentration inequalities}
We need the following two basic concentration inequalities. We will refer to the following as Chernoff's bound.
\begin{lemma}[Chernoff bound]\label{chernoff} Let $X:=\sum_{i=1}^m X_i$ where $(X_i)_{i\in[m]}$ is a sequence of independent indicator random variables with $\mathbb{P}(X_i=1)=p_i$. Let $\mathbb{E}[X]=\mu$. Then, for any $0<\gamma<1$, we have that $\mathbb{P}(|X-\mu|\geq \gamma \mu)\leq 2e^{-\mu \gamma^2/3}$.
\end{lemma}

We use the following corollary of Chernoff's bound often: that if $R$ is a $p$-random subset of an $n$-element set, then with high probability we have that $|pn-|R||\leq \log n\sqrt n$.

Sometimes the random variables we consider have slight dependencies. In this case, we rely on Azuma's inequality which we now cite. Given a product probability space $\Omega = \prod_{i\in[n]} \Omega_i$, a random variable $X\colon \Omega\to \mathbb{R}$ is called $C$-\textbf{Lipschitz} if $|X(\omega)-X(\omega')|\leq C$ whenever $\omega$ and $\omega'$ differ in at most $1$-coordinate.
\begin{lemma}[Azuma's inequality]
Let $X$ be $C$-Lipschitz random variable on a product probability space with $n$ coordinates. Then, for any $t>0$,
$$\mathbb{P}(|X-\mathbb{E}(X)|> t)\leq 2e^{\frac{-t^2}{nC^2}}.$$
\end{lemma}

\subsubsection{Nibble-type lemmas}

\par We say that a $r$-partite $r$-uniform hypergraph $H$ is $(\gamma, p, n, k)$-\textbf{regular} if every part has $(1\pm \gamma)n$ vertices and every vertex has degree  $(1\pm \gamma)pn^k$. 
For a $3$-uniform $3$-partite hypergraph $H$, vertices $u,v$ and a subset $U\subseteq V(H)$, we define the \textbf{pair degree} of $(u,v)$ into $U$ as the number of vertices in $U$ which are in the neighbourhood of both $u$ and $v$, i.e. the number of vertices $z$ in $U$ such that there exists $y,w\in V(H)$ such that $\{u,z,y\}$ and $\{v,z,w\}$ are both edges of $H$. A $3$-uniform $3$-partite hypergraph $H$ is $(\gamma, p, n)$-\textbf{typical} if it is $(\gamma, p, n, 1)$-\textbf{regular} and every pair of vertices $u,v$ coming from the same part has pair degree $(1\pm \gamma )p^2n$ into every other part of $H$.
A hypergraph is \textbf{linear} if it maximum co-degree at most $1$.
\par The following nibble-type result due to Ehard, Glock, and Joos is convenient to use for our application here.
\begin{theorem}[\cite{ehard_glock_joos_2020}]\label{thm:hypermatchingehard}
Suppose $\delta\in(0,1)$ and $r\in \bN$ with $r\ge 2$, and let $\eps:=\delta/50r^2$. Then there exists $\Delta_0$ such that for all $\Delta\ge \Delta_0$, the following holds.
Let $\cH$ be an $r$-uniform hypergraph with $\Delta(\cH)\leq \Delta$ and $\Delta^c(\cH)\le \Delta^{1-\delta}$ as well as $e(\cH)\leq \exp(\Delta^{\eps^2})$.
Suppose that $\cW$ is a set of at most $\exp(\Delta^{\eps^2})$ weight functions on~$E(\cH)$.
Then, there exists a matching $\cM$ in~$\cH$ such that $\omega(\cM)=(1\pm \Delta^{-\eps}) \omega(E(\cH))/\Delta$ for all $\omega \in \cW$ with $\omega(E(\cH))\ge \max_{e\in E(\cH)}\omega(e)\Delta^{1+\delta}$.
\end{theorem}
Applying Theorem~\ref{thm:hypermatchingehard} with a single uniform weight function, we obtain the following.
\begin{corollary}\label{cor:nibble}
Let $n$ be sufficiently large and let $p\geq n^{-1/10000}$.
\begin{enumerate}
    \item Let $\mathcal{H}$ be a $6$-uniform $6$-partite hypergraph on $n$ vertices which is $(n^{-0.01},p,n,2)$-regular with maximum co-degree at most $10n$. Then, $\mathcal{H}$ has a matching covering all but $n^{1-10^{-5}}$ vertices.
    \item For any $\gamma\geq 0$, every $(\gamma, \delta, n)$-regular linear tripartite hypergraph has a matching covering  all but at most $n^{1-1/500}+3\gamma n$ vertices.
\end{enumerate}
\end{corollary}
The below lemma allows us to incorporate some non-random vertices/colours into the nibble process. It unfortunately does not directly imply Lemma~\ref{lem:deterministicnibble}, but it will be an important ingredient in its proof.
\begin{lemma}[\cite{muyesser2022random}]\label{lem:directedwithdeterministic}
Let $a,b,c\geq n^{-1/10^{-20}}$ and let $\ell\in \mathbb{N}$ be such that $\ell \geq \max \{an,bn,cn\}-n^{0.7}$ and letting $(x,y,z)=(\ell - an, \ell -bn, \ell -cn )$ suppose that we have $x+y\leq cn/2$, $x+z\leq bn/2$ and $y+z\leq an/2$. Let $A,B,C\subseteq G$ be $a,b,c$-random subsets of $G$ respectively, sampled with $A$ and $B$ disjoint, and $C$ independent of $A,B$. Let $k:=\lfloor n^{1-10^{-5}} \rfloor $. Then, with probability at least $1-1/n$ the following holds.

\par Let $A', B', C'\subseteq G$ with $|B\setminus B'|, |A\setminus A'|, |C\setminus C'|\leq n^{0.78}$, $|C'|-k=|A'|=|B'|=\ell$. Then, there is a perfect directed $C'$-matching in $\vec{K}_G[A',B'; C']$.
\end{lemma}

\subsection{Group theoretic tools}\label{sec:grouptheoretictools}
Given a sequence $\vec{c}=(c_1, c_2,\ldots, c_k)$ of group elements, and another group element $v$, we define the following sequences:
\begin{itemize}
    \item $P_{out}(v, \vec{c}):=(v, v-c_1, v-c_1-c_2, \ldots, v-c_1-c_2-\cdots- c_k)$
    \item $P_{in}(v, \vec{c}):=(v, v+c_1, v+c_1-c_2, \ldots, v+c_1+c_2+\cdots + c_k)$
\end{itemize}
Observe that in $\vec{K}_G$, $P_{out}(v, \vec{c})$ denotes the vertex sequence obtained by starting a walk from $v$, and following the out-edges given by the sequence $\vec{c}$. $P_{in}(v, \vec{c})$ is analogous, except it follows the in-edges.
\par We call a sequence of group elements $ \vec{c}=(c_1, c_2,\ldots, c_k)$ a \textbf{path}-\textbf{candidate} if the partial sums $\sum_{i\in[j]} c_i$ for each $j\in [k]$ (including $j=0$) are all distinct. Equivalently, all non-empty partial sums (of consecutive elements) are non-zero. Observe that $ \vec{c}$ being a path-candidate simply means that for any vertex $v\in G$, the walks $P_{out}(v, \vec{c})$ and $P_{in}(v, \vec{c})$ both give paths in $\vec{K}_G$.
\par We call a sequence of group elements $\vec{c}=(c_1, c_2,\ldots, c_k)$ a \textbf{cycle}-\textbf{candidate} if $(c_1, c_2,\ldots, c_{k-1})$ is a path-candidate and $\sum_{i\in[k]} c_i = 0$. This means that $P_{out}(v, \vec{c})$ and $P_{in}(v, \vec{c})$ both give cycles (of length $k$) in $\vec{K}_G$.
\par We call a sequence of group elements \textbf{rainbow} if all coordinates are distinct. Notice that a necessary condition for solving the FGT conjecture is a partition of $G\setminus\{0\}$ into rainbow cycle-candidates, each of length $k$.
\par The following two definitions only come up in the cover-down strategy for $k\geq 10$.
\par We call a collection of length $k$ sequences \textbf{dissociable} if for any two distinct sequences $\vec{c}=(c_1, c_2,\ldots, c_k)$ and $\vec{b}=(b_1, b_2,\ldots, b_k)$ and $j,j'\in [k]$, $\sum_{i\in[j]} c_i\neq \sum_{i\in[j']} b_i$. This means that $P_{out}(v,\vec{c})$ and $P_{out}(v,\vec{b})$ are disjoint except on $v$. We call such a collection \textbf{near-dissociable} if the previous property holds for each $j,j'\leq k-1$ (or equivalently, the sequences obtained by removing the last element from each tuple gives a dissociable family). This means that the corresponding directed walks are disjoint except on the endpoints.
\par We call two length $k$ sequences $(c_1, c_2,\ldots, c_k)$ and $\vec{b}=(b_1, b_2,\ldots, b_k)$ \textbf{separable at distance} $d$ if for all $j,j'\in [k]$, $\sum_{i\in [j]}c_i + \sum_{i\in [j']} b_i \notin \{-d,d\}$. This means that
for $v$ and $w$ where $v-w=d$, $P_{out}(v,\vec{c})$ and $P_{in}(w,\vec{b})$ are disjoint (except potentially on $v$ or $w$).

The following simple lemma is key to the gadget finding strategy presented in Section~\ref{freeproducts}.

\begin{lemma}\label{lem:manymaps}
Let $G$ be abelian of order $n$. Then either the map $x\to 2x$ or the map $x\to 3x$ has an image of size at least $n^{1/5}$.
\end{lemma}
\begin{proof} For an integer $k$ and an abelian group $G$ let $G_k$ denote the number of distinct images of the map $G\to G$ via $x\to k\cdot x$. Observe that $(G\times H)_k=G_k\cdot H_k$. By the fundamental theorem of finite abelian groups, $G\cong (\mathbb{Z}_2)^\ell \times (\mathbb{Z}_3)^j \times H$ where $H$ is a product of cyclic groups none of which are of order $2$ or $3$. Note that $(\mathbb{Z}_t)_2\geq t/2$ if $t\neq 2$ and $(\mathbb{Z}_t)_3\geq t/3$ if $t\neq 3$. This, together with the observation, implies that $H_2 \geq |H|/2^{\log_4 |H|}=|H|^{0.5}$ and $H_3 \geq |H|/3^{\log_4 |H|}\geq |H|^{0.2}$ (using that $H$ has at most $\log_4 |H|$ many cyclic factors in its decomposition as each factor needs to have size at least $4$). Also, $(\mathbb{Z}_2)^\ell_3=2^\ell$ and $(\mathbb{Z}_3)^j_2=3^j$, as multiplying by $3$ in $\mathbb{Z}_2$ and multiplying by $2$ in $\mathbb{Z}_3$ are both bijections.
\par Note that $|G|=n=2^\ell \cdot 3^j \cdot |H|$. If $j\geq \ell$, we have that $n\leq 6^{j}|H|$ and $|G_2|\geq 3^j |H|^{0.5} $, so $|G_2|\geq n^{1/2}$. Otherwise, if $j<\ell$, we have that $n\leq 6^\ell |H|$ and $|G_3|\geq 2^\ell |H|^{0.2}$, so $|G_3|\geq n^{1/5}$.
\end{proof}

\subsubsection{Finding gadgets}\label{freeproducts}
In this section, we adapt some tools from Section 3.6 of \cite{muyesser2022random} and refine them in the setting of $\vec{K}_G$ where $G$ is an abelian group. The main result of the section is Lemma~\ref{lem:patternfinding}, which is key to our gadget finding strategy.
\par By $F_k$, we denote the free abelian group on $k$ generators, the free variables are denotes as $v_1, \dots, v_k$ (recall that $F_k\cong \mathbb{Z}^k$). $G\ast  F_k$ denotes the free product, and $(G\ast  F_k)^{\mathrm{ab}}$ denotes the abelianization of the free product (recall that the abelianization $G^{\mathrm{ab}}$ of a group $G$ is defined by the property that any homomorphism $G\to H$ where $H$ is abelian factors uniquely through $G^{\mathrm{ab}}$). A \textbf{word} is simply an element of $(G\ast  F_k)^{\mathrm{ab}}$. As all groups $G$ are abelian in this paper, $(G\ast  F_k)^{\mathrm{ab}}\cong G\times F_k$, where the latter denotes a direct product. However, the former perspective makes it clear that each word $w$ can be represented as
$$w=z_1\cdot v_1+\cdots +z_t\cdot v_t+g$$
where each $v_i$ is a free variable, each $z_i$ is a (non-zero) integer, and $g\in G$, and this representation is unique up to reordering the summands.
\par A word $w$ is \textbf{constant} if $w\in G$, i.e. $w$ does not include any free variables. We say that $z_i$ is the \textbf{coefficient} of $v_i$. We say that $w$ is \textbf{linear in} $v_i$ if $z_i\in\{1,-1\}$. We say that $w$ is \textbf{linear} if each $z_i\in \{1,-1\}$, and $w$ is not constant. That is, $w$ is linear in each free variable, and there exists at least one free variable in $w$.
\par A homomorphism $\pi:(G\ast  F_k)^{\mathrm{ab}}\to G$ is a \textbf{projection} if $\pi(g)=g$ for all $g\in G$. We show two basic properties of projections. We remind the reader that throughout, $G$ is a finite abelian group of order $n$.
\begin{lemma}\label{Lemma_free_extension_universal_property}
For each function $f:\{v_1, \dots, v_k\}\to G$, there is precisely one projection $\pi_f:(G\ast  F_k)^{\mathrm{ab}}\to G$ which agrees with $f$ on $\{v_1, \dots, v_k\}$. In particular, there are precisely $n^k$ projections $(G\ast  F_k)^{\mathrm{ab}}\to G$.
\end{lemma}
\begin{proof}
    By the universal property of free abelian groups, there is a unique homomorphism $g\colon F_k\to G$ which agrees with $f$ on $\{v_1, \dots, v_k\}$. By the universal property of free products, there is a unique homomorphism $h\colon G\ast  F_k \to G$ that agrees with $g$ on $F_k$ and with the identity homomorphism $G\to G$. As $G$ is abelian, $h$ can be written uniquely as $h=h'\circ p$ where $p\colon (G\ast  F_k)\to (G\ast  F_k)^{\mathrm{ab}}$ is the quotient map, and $h'\colon (G\ast  F_k)^{\mathrm{ab}}\to G$ is a projection that agrees with $f$ on $\{v_1, \dots, v_k\}$. This gives the desired one to one correspondence.
\end{proof}
\begin{lemma}\label{Lemma_count_projections_fixing_one_image}
Let $w\in \{v_1, \dots, v_k\}$ be linear in some free variable $v_i$ and let $g\in G$. Then there are exactly $n^{k-1}$ projections $\pi:(G\ast F_{k})^{\mathrm{ab}}\to G$ having $\pi(w)=g$.
\end{lemma}
\begin{proof}
    Suppose that $i=k$, without loss of generality. By Lemma~\ref{Lemma_free_extension_universal_property} there are exactly $n^{k-1}$ projections $\pi:(G\ast F_{k-1})^{\mathrm{ab}}\to G$. For each such $\pi$, we show that there is a unique projection $\pi'$ that agrees with $\pi$ and additionally has $\pi'(w)=g$. By linearity of $w$ in $v_k$, the equation $w=g$ rearranges into $v_k=h$ for some $h\in (G\ast F_{k})^{\mathrm{ab}}$ and $v_k$ does not appear in $h$. So, $\pi'(w)=g$ is equivalent to $\pi'(w)=\pi'(g)$ (as $\pi'$ is a projection) which is equivalent to $\pi'(v_k)=\pi'(h)$ (as $\pi$ is a homomorphism). As $\pi'$ agrees with $\pi$ and $h\in (G\ast F_{k-1})^{\mathrm{ab}}$, we have that $\pi'(v_k)=\pi(h)\in G$. Therefore, $\pi'$ has that $\pi'(v_k)=\pi(h)$, and $\pi'(v_i)=\pi(v_i)$ for $1\leq i<k$. By Lemma~\ref{Lemma_free_extension_universal_property}, there is a unique projection with this property.
\end{proof}
The following is a simple consequence of the previous lemma.
\begin{lemma}\label{Lemma_upper_bound_set_of_linear_words}
Let $S\subseteq (G\ast F_{k})^{\mathrm{ab}}$ be a set of elements which are each linear in at least one variable, and let $U\subseteq G$. Then the number of projections $\pi:(G\ast F_{k})^{\mathrm{ab}}\to G$  for which $\pi(S)$ intersects $U$ is $\leq |S||U|n^{k-1}$.
\end{lemma}

\begin{definition}
 Let $w,w'\in (G\ast F_{k})^{\mathrm{ab}}$. We say that $w$ and $w'$ are \textbf{separable} if any of the following hold.
\begin{enumerate}[label=(\alph*)]
\item\label{linear} $w'-w$ is linear in some free variable $v_i$. Note that this is equivalent to asking that there exists a free variable $v$ with coefficient $z$ in $w$ and $z'$ in $w'$ and we have $|z-z'|=1$.
\item\label{gnotzero} The equation $w=w'$ rearranges into $g=0$ for some non-zero group element $g\in G$.
\item\label{32} The equation $w=w'$ rearranges into $3v_i-2v_j=g$ for some group element $g\in G$ and distinct free variables $v_i$ and $v_j$.
\end{enumerate}
\end{definition}


\begin{definition}
Let $S\subseteq G\ast F_k$. We say that a homomorphism $\phi:(G\ast F_{k})^{\mathrm{ab}}\to G$, \textbf{separates} $S$ if for every separable $w,w'\in S$ we have $\phi(w)\neq \phi(w')$.
\end{definition}

\begin{lemma}\label{Lemma_lower_bound_separated_set}
Let $n\geq 10^{100}$.
Let $S\subseteq(G\ast F_{k})^{\mathrm{ab}}$ be a  set of size $\leq 1000$. Then there are at most $|S|^2 n^{k-1/5}$ projections $\pi:(G\ast F_{k})^{\mathrm{ab}} \to G$ which do not separate $S$.
\end{lemma}
\begin{proof}
Let $w$ and $w'$ be two separable words in $S$. We case on which of the conditions (a)/(b)/(c) makes $w$ and $w'$ separable, and count the projections which do not separate them in each case.
\begin{itemize}
    \item[\ref{linear}] In this case, $\pi(w)=\pi(w')$ is equivalent to $\pi(w-w')=e$ (using that $\pi$ is a homomorphism). By Lemma~\ref{Lemma_count_projections_fixing_one_image}, there are $n^{k-1}$ projections $\pi$ satisfying this latter identity.
    \item[\ref{gnotzero}] We can rearrange $\pi(w)=\pi(w')$ into $\pi(g)=\pi(e)$ using that $\pi$ is a homomorphism. The latter implies that $g=e$ using that $\pi$ is a projection, which is a contradiction. Hence there can be no projections $\pi$ with $\pi(w)=\pi(w')$ in this case.
    \item[\ref{32}] Similarly to the previous cases, we can rearrange $\pi(w)=\pi(w')$ into $3\pi(v_i)-2\pi(v_j)=g$.   Using Lemma~\ref{lem:manymaps}, suppose first that $x\to 3x$ has at least $n^{1/5}$ images, and suppose $\pi$ also satisfies $\pi(v_j)=g'$ for some $g'\in G$, so we have $3\pi(v_i)=g''\in G$ (where $g''=g+2g'$). As $x\to 3x$ is a homomorphism ($G$ is abelian) and has at least $n^{1/5}$ images, the preimage of $g''$ under the map $x\to 3x$ has size at most $n^{4/5}$ (each non-empty preimage must have the same size in a group homomorphism). This means $\pi(v_i)$ must live in a set $T_{g'}$ of size at most $n^{4/5}$ assuming that $\pi(w)=\pi(w')$ and $\pi(v_j)=g'$. Thus, if $\pi(w)=\pi(w')$, $\pi$ must agree with one of $n^{k-1}n^{4/5}$ functions $f\colon \{v_1,\ldots, v_k\}\to G$, meaning that there are at most $n^{k-1/5}$ such projections, using Lemma~\ref{Lemma_free_extension_universal_property}. A symmetric argument works when $x\to 2x$ has at least $n^{1/5}$ images.
\end{itemize}
As there are at most $\binom{|S|}{2}$ pairs of separable words in $S$, the desired bound follows.
\end{proof}

\begin{lemma}\label{Lemma_disjoint_separating_projections}
Let $n\geq 10^{10}$, and let $S\subseteq (G\ast F_{k})^{\mathrm{ab}}$ be a set of at most $100$ elements which are all linear in at least one variable. Then, there are projections $\pi_1, \dots, \pi_{n/200}$  which separate $S$ and have $\pi_1(S), \dots, \pi_{n/200}(S)$  disjoint.
\end{lemma}
\begin{proof} Call a projection \textit{good} if it separates $S$. Let $\pi_1,\ldots, \pi_t$ be a maximal collection of good projections with the sets $\pi_i(S)$ being pairwise disjoint. Set $T=\pi_1(S)\cup \cdots \cup \pi_t(S)$ noting $|T|=|S|t$. For any good projection $\pi$, we must have $\pi(S)\cap T\neq \emptyset$ by maximality, so by Lemma~\ref{Lemma_upper_bound_set_of_linear_words}, we have that there are at most $100tn^{k-1}$ good projections. On the other hand, there are at most $|S|^2n^{k-1/5}\leq n^{k}/2$ projections which are not good by Lemma~\ref{Lemma_lower_bound_separated_set}, so there are at least $n^k/2$ good projections (there are $n^k$ projections total). Combining, we have $n^k/2\leq 100tn^{k-1}$, meaning $t\geq n/200$, as desired.
\end{proof}

Combining the previous lemma with a standard application of Chernoff's bound, we obtain the following.

\begin{lemma}\label{lem:gadgetfinding}
Let $p \geq n^{-1/700}$. Let $R$ be $p$-random subset of $G$. With high probability, the following holds.
\par Let $S\subseteq (G\ast F_{k})^{\mathrm{ab}}$ a set of $\leq 100$ elements which are each linear in at least one variable, and let $U\subseteq G$ with $|U|\leq p^{100}n/1000$. Then there is a projection $\pi: (G\ast F_k)^{\mathrm{ab}} \to G$ which separates $S$, has $\pi(S)\cap U=\emptyset$ and $\pi(S)\subseteq R$.
\end{lemma}
We now package everything we have so far into a lemma (Lemma~\ref{lem:patternfinding}) that fits nicely with our application in the setting of $\vec{K}_G$.

\begin{definition}\label{def:pattern} Given a group $G$,
  a \textbf{pattern} $P$ is a directed (simple) graph equipped with a vertex and edge labelling $\phi$ with the following properties.
  \begin{enumerate}
      \item $\phi$ maps vertices and edges to $(G\ast F_{k})^{\mathrm{ab}}$ for some positive integer $k$.
      \item Each vertex gets a distinct label via $\phi$ (i.e. $\phi|_{V(P)}$ is injective, but distinct edges can potentially receive the same label)
      \item If $\vec{e}\in E(P)$ is a directed edge from $v$ to $w$ for $v,w\in V(P)$, we have that $\phi(v)-\phi(w)=\phi(\vec{e})$.
  \end{enumerate}
\end{definition}
\par In Figure~\ref{fig:gadgets} we have several examples of patterns. We can naturally view the edge-labels as colours, hence each pattern can also be viewed as an edge-coloured graph.

A \textbf{pairwise separable} subset $S$ is a subset where any two distinct words $w$ and $w'$ are separable.
\begin{definition}
We call a pattern $(P,\phi)$ \textbf{well-distributed} if the following two conditions hold.
\begin{enumerate}
    \item The subsets (viewed as sets, not multisets) $\{\phi(v)\colon v\in V(P)\}$ and $\{\phi(\vec{e})\colon \vec{e}\in E(P)\}$ are both pairwise separable subsets of $G\ast F_k$.
    \item Each label is either a constant, or linear in at least one free variable.
\end{enumerate}
\end{definition}
\par Notice that we are not insisting that any $\phi(v)$ and $\phi(\vec{e})$ are separable for a vertex $v$ and edge $\vec{e}$. This is because in our applications vertex sets and colour sets are sampled independently, hence we don't need any separability properties.
\begin{definition}
    A \textbf{copy} of a well-distributed pattern $(P,\phi)$ is a subgraph $S$ of $\vec{K}_G[V;C]$ such that there exists a projection $\pi\colon (G\ast F_k)^{\mathrm{ab}} \to G$ (where $k$ is the number of free variables used in $\phi$) with the following properties.
    \begin{enumerate}
        \item $\pi$ maps $\phi[V(P)]$ (the vertex labels) to $V(S)\subseteq V$ and $\phi[E(P)]$ (the edge labels) to $C$.
        \item $\pi$ separates $\phi[V(P)]$ and $\pi$ separates $\phi[E(P)]$. In particular, $\pi$ is injective when restricted to $\phi[V(P)]$.
        \item The $(v,w)$ is a directed edge of $S$ if and only if there is a directed edge from the vertex with the label $\pi^{-1}(v)$ to the vertex with the label $\pi^{-1}(w)$ in $P$.
    \end{enumerate}
\end{definition}
An \textbf{edge-coloured directed graph isomorphism} $\psi$ between two edge-coloured simple directed graphs $G_1,G_2$ is a graph isomorphism mapping vertices of $G_1$ to vertices of $G_2$ and mapping edges $(v,w)$ of $G_1$ to $(\psi(v),\psi(w))$ that preserves the direction of each edge and respects colours. This means that $e_1$ and $e_2$ of $G_1$ have the same colour if and only if $\psi(e_1)$ and $\psi(e_2)$ have the same colour.
\begin{observation}\label{patternsaresubgraphs} Let $S$ be a copy of $P$. Then, there is a edge-coloured directed graph isomorphism $\psi$ between $P$ and $S$. Furthermore, if $x$ is the label of a vertex or colour of $P$, and $x$ is a constant, $\psi(x)=x$.
\end{observation}
\begin{proof}
    The projection $\pi$ witnessing that $S$ is a copy naturally corresponds to a $\psi$ with the desired properties, as $\pi$ fixes elements of $(G\ast F_k)^{\mathrm{ab}}$ which are constants by definition of a projection.
\end{proof}
The following is a consequence of the definition of well-distributed, copy, and applying Lemma~\ref{lem:gadgetfinding} to $R_1$ and $R_2$.
\begin{lemma}\label{lem:patternfinding}
Let $p\geq n^{-1/700}$. Let $R_1$ and $R_2$ be $p$-random subsets of $G$, sampled independently. With high probability, the following holds.
\par Let $P$ be a well-distributed pattern with $V(P)+E(P)\leq 100$. Let $U\subseteq G$ with $|U|\leq p^{150}n/10^{7}$. Let $V'$ and $C'$ be the set of labels of vertices and colours in $P$ which are constants. Then, there is a copy of $P$ in $\vec{K}_G[(R_1\setminus U)\cup V';(R_2\setminus U)\cup C']$.
\end{lemma}
As an example application of the above result, we recommend the reader to inspect the proof of Lemma~\ref{lem:pathcyclecandidates}.

\subsubsection{Partitioning into sets with fixed sum}
 In this subsection, we prove some lemmas designed to ``cover-down'' part of the absorption strategy. See the proof overview for more context. For the reader interested in the $k=3$ case, the $k=3$ case of Lemma~\ref{lem:generalisedtannenbaum} is all that is required, and this case follows directly from Lemma~\ref{Lemma_zero_sum_equipartition} (without having to use Lemma~\ref{alspach}). We cite the following three results from \cite{muyesser2022random}.

\begin{theorem}[\cite{muyesser2022random}]\label{thm:maintheoremv1} Let $p\geq n^{-1/10^{100}}$. Let $G$ be an abelian group of order $n$. Let $R^1,R^2\subseteq G$ be disjoint $p$-random subsets, and let $R^3\subseteq G$ be a $p$-random subset, sampled independently with $R^1$ and $R^2$. Then, with high probability, the following holds.
\par Let $X,Y,Z$ be equal-sized subsets of $G_A$, $G_B$, and $G_C$ respectively, satisfying the following properties.
\begin{itemize}
    \item $|(R^1_A\cup R^2_B\cup R^3_C) \Delta (X\cup Y\cup Z) |\leq p^{10^{10}}n/\log(n)^{10^{10}}$
    \item $\sum X+\sum Y + \sum Z = 0$
    \item Suppose that $0\notin X\cup Y\cup Z$
\end{itemize}
 Then, $H_G[X,Y,Z]$ contains a perfect matching.
\end{theorem}

\begin{lemma}[\cite{muyesser2022random}]\label{Lemma_zero_sum_equipartition}
Let $p\geq n^{-1/10^{100}}$ and $3\leq k\leq 100$. Let $R$ be a $p$-random subset of an abelian group $G$. With high probability the following holds.

Let $X\subseteq G$ with $|X\triangle R|\leq  p^{10^{10}}n/\log(n)^{10^{18}}$, $0\notin X$, $\sum X=0$, and $|X|\equiv 0\pmod k$. Then, $X$ can be partitioned into zero-sum sets of size $k$.
\end{lemma}

\begin{lemma}[\cite{muyesser2022random}]\label{Lemma_find_set_with_correct_sum}
Let $p\geq n^{-1/700}$ and let $R$ be a $p$-random subset of an abelian group $G$. With high probability the following holds.

Let $\epsilon\in [2\log n/\sqrt n, p^{800}/10^{4010}]$. For any $m$ with $|m-pn|\leq  \epsilon n$, $g\in G$ and $Z$ with $|Z|\geq m+3$, $|R\setminus Z|\leq  \epsilon n$, there is a set $R'\subseteq Z$ with $|R'|=m$, $|R'\triangle R|\leq 6\epsilon n$, and $\sum R'=g$.
\end{lemma}

\begin{corollary}\label{lem:differentsum}
Let $p\geq n^{-1/10^{100}}$. Let $R$ be a $p$-random subset of an abelian group $G$. With high probability the following holds.

Let $X\subseteq G$ with $|X\triangle R|\leq  p^{10^{10}}n/\log(n)^{10^{22}}$. Let $\alpha\in G$. $0 \notin X$, $|X|\equiv 0\pmod 4$, $\sum X=(|X|/4)\cdot \alpha$. Then, $X$ can be partitioned into sets of size $4$ with sum $\alpha$.
\end{corollary}
\begin{proof}
Let $R_1,R_2,R_3,R_4$ be disjoint $(p/4)$-random subsets of $G$ which partition $R$. Let $S$ be a $(p/4)$-random subset of $G$, sampled independently with the previous sets. Note that the set $-S-\alpha$ is also a $(p/4)$-random subset of $G$, which is independent with the previous random sets (not including $S$). With high probability, Theorem~\ref{thm:maintheoremv1} holds with the sets $(R_1,R_2, S)$ and $(R_3,R_4, -S-\alpha)$, Lemma~\ref{Lemma_find_set_with_correct_sum} holds for each random set, and by Chernoff's bound, each random set is within a $n^{0.6}$ term of its expectation.
\par  Let $X\subseteq G$ be given. By Lemma~\ref{Lemma_find_set_with_correct_sum}, we can partition $X$ into equal sized sets $X_1,X_2,X_3,X_4$ such that $|X_i\Delta R_i|\leq p^{10^{10}}n/\log(n)^{10^{21}}$ and $\sum X_1 +\sum X_2= 0$. This readily implies that $\sum X_3 +\sum X_4= (|X|/4)\cdot \alpha$ by the sum condition on $X$. Similarly, via Lemma~\ref{Lemma_find_set_with_correct_sum}, we can fix a set $S'$ with $\sum S' = 0$, $|S'|=|X|/4$, and such that $S'$ has small symmetric difference with $S$. This implies that $\sum (-S'-\alpha)= -(|X|/4)\cdot \alpha$, and also we have that $-S'-\alpha$ has small symmetric difference with $-S-\alpha$. Thus we have that $\sum X_1 +\sum X_2 + \sum S'=0$ and $\sum X_3 +\sum X_4 + \sum (-S'-\alpha)=0$. Also, we remark that $S'$ can be chosen so that both $S'$ and $-S'-\alpha$ do not contain $0$. So we can apply Theorem~\ref{thm:maintheoremv1} twice to deduce that both $H_G[X_1,X_2,S']$ and $H_G[X_3,X_4,-S'-\alpha]$ has a perfect matching. For each $s'\in S'$, consider the edges $(x_1,x_2,s')$ and $(x_4,x_4,-s'-\alpha)$ guaranteed by the two perfect matchings, and observe that $x_1+x_2+x_3+x_4=\alpha$. Combining $4$-tuples of this form, we obtain the desired partition of $X$.
\end{proof}

For technical reasons, our absorption strategy for large $k$ requires the assumption that $k\geq 10$. This leaves the case of $3\leq k\leq 9$ open. The previous lemmas already give us a way to partition sets into $k$-sets which are zero sum in this regime. Once we have access to such a partition, a natural strategy is to look for an ordering of the $k$-set yielding a cycle-candidate, in order to be able to perform the cover-down step (see proof overview). We rely on the following result of Alspach and Liversidge to find suitable orderings. Similar results for cyclic groups were obtained in \cite{costa2020some, hicks2019distinct}.

\begin{lemma}[Alspach-Liversidge, \cite{alspach2020strongly}, Corollary 5.2]\label{alspach}
    Let $G$ be any abelian group (not necessarily finite). Let $S\subseteq G$ be of size at most $9$.  with $\sum S=0$. Then, $S$ admits an ordering yielding a rainbow cycle-candidate if $\sum S=0$, and otherwise $S$ admits an ordering yielding a rainbow path-candidate.
\end{lemma}

We can now prove the main lemma of this section.

\begin{lemma}\label{lem:generalisedtannenbaum}
There exists an absolute constant $\eps_{\ref{lem:generalisedtannenbaum}}$ such that the following holds. Let $3\leq k \leq 9$. Let $G$ be an abelian group of order $n$, let $p\geq n^{-\eps_{\ref{lem:generalisedtannenbaum}}}$. Let $R$ be a $p$-random subset of $G$. With high probability, the following holds. Let $R'\subseteq G$ such that $|R'\Delta R|\leq p^{10^{10}}n/\log(n)^{10^{23}}$. Suppose $k$ divides $|R'|$ and that $0\notin R'$.
\begin{enumerate}
    \item Suppose that $\sum R'=0$. Then, $R'$ can be partitioned into $k$-tuples which are rainbow cycle-candidates.
    \item Suppose that $k=4$ and for some $\alpha\in G\setminus\{0\}$, $\sum R'=(|R'|/k)\cdot \alpha $. Then, $R'$ can be partitioned into $k$-tuples which are rainbow path-candidates with sum $\alpha$.
\end{enumerate}
\end{lemma}
\begin{proof}
Choose some $\eps_{\ref{lem:generalisedtannenbaum}}\leq 10^{-1000}$. With high probability, Lemma~\ref{Lemma_zero_sum_equipartition}
and Corollary~\ref{lem:differentsum} both hold for $R$. Let $R'$ be given. For part (1), we apply Lemma~\ref{Lemma_zero_sum_equipartition} to partition $R'$ into $k$-sets which are zero-sum. Then, Lemma~\ref{alspach} implies that each of these $k$ sets can be ordered to obtain a rainbow cycle candidate, as $k\leq 9$. For part (2), we apply Corollary~\ref{lem:differentsum} to partition into $4$-tuples each with sum $\alpha$. As $\alpha\neq 0$, we can order each tuple to be path-candidates by Lemma~\ref{alspach}. This concludes the proof.
\end{proof}

\subsubsection{Good families of colours}
In this section we have some lemmas designed to deal with the $k\geq 10$ case of the cover-down step.
\begin{lemma}\label{Lemma_counting_good_tuples}
    Let $G$ be an abelian group of order $n$, let $s\in G\setminus\{0\}$ let $\bar{T}$ be the collection of $k$-tuples $(g_1,\ldots, g_k)$ with $\sum g_i=s$ (note $|\bar{T}|=n^{k-1}$). Suppose $n^{0.01}\geq k\geq 2$ and $n\geq 10^{10}$. Let $S$ be a subset of $G$ of size at most $n/(20k)$. Then, all but at most $n^{k-1}/4$ tuples in $\bar{T}$ are all rainbow path-candidates disjoint with $S$.
\end{lemma}
\begin{proof}
\par As $s\neq 0$, we can count that there are at least
$(n-1)(n-2)\cdots (n-k+1)\geq (n-k)^{k-1}\geq n^{k-1}/1.01$ (using that $k$ is small for the final inequality) path-candidates in $\bar{T}$.
\par If $k\geq 3$, by a direct counting we can see that there are at most $k^2nn^{k-3}\leq k^2n^{k-2}$ tuples in $\bar{T}$ with two coordinates being equal. If $k=2$, using that $G$ is an abelian group and $s\neq 0$, we see that there are at most $n/2$ tuples (generously) in $\bar{T}$ with two coordinates being equal. In either case, all but $n^{k-1}/2$ tuples in $\bar{T}$ are rainbow (using that $n\gg k$).
\par For each $g\in G$, there are at most $kn^{k-2}$ elements of $\bar{T}$ having $g$ in some coordinate, here we used that $k\geq 2$. So, there are at most $|S|kn^{k-2}\leq n^{k-1}/20$ many tuples $\bar{t}$ not disjoint with $\mathcal{S}_G$.
\par We derive that there are at least $n^{k-1}/1.01-n^{k-1}/2-n^{k-1}/20\geq n^{k-1}/4$ tuples in $\bar{T}$ satisfying all the desired properties.
\end{proof}
\begin{lemma}\label{lem:biggroupspecialfamily} There exists some absolute constant $C_{\ref{lem:biggroupspecialfamily}}$ such that the following holds. Let $G$ be an abelian group of order $n$, let $n\geq 10^{100}$, and let $k$ be a integer such that $10\leq k\leq n^{0.001}$. Then, $G$ contains two families $\mathcal{F}_G=\mathcal{F}_G(k)$ and $\mathcal{S}_G=\mathcal{S}_G(k)$ of disjoint tuples $\mathcal{F}_1,\ldots, \mathcal{F}_{\lfloor n/kC_{\ref{lem:biggroupspecialfamily}}\rfloor}$ and $\mathcal{S}_1,\ldots, \mathcal{S}_{\lfloor n/kC_{\ref{lem:biggroupspecialfamily}}\rfloor}$ with the following properties.
\begin{enumerate}[label=(\arabic*)]
    \item Each $\mathcal{F}_i$ is of size $4$ and has the same sum $f=f(G,k)$.
    \item\label{si1} Each $\mathcal{S}_i$ has the same size and sum $s=s(G,k)$. In fact, $|\mathcal{S}_i|=:z_\mathcal{S}\in\{2,3,4,5\}$.
    \item\label{si2} Each $\mathcal{S}_i$ is a rainbow path candidate, and $\mathcal{S}_G$ is near-dissociable.
    \item $k-4-z_\mathcal{S}$ is divisible by $4$. Furthermore, set $q:=q_{G,k}=-((k-4-z_\mathcal{S})/4)f-s$. We have that $q\neq 0$.
    \item Each $\mathcal{F}_i$ can be partitioned into two tuples, $\mathcal{F}_i^+=(f_i^{+,1}, f_i^{+,2})$ and $\mathcal{F}_i^-=(f_i^{-,1}, f_i^{-,2})$, both of which are rainbow path candidates. The resulting collection of $\mathcal{F}_i^+$ and $\mathcal{F}_i^-$ are both dissociable. Also, each $\mathcal{F}_i$ is a rainbow path candidate, and $\mathcal{F}_G$ is near-dissociable.
    \item For each $m\in \{0,1,2,\ldots, k\}$ and $i\in \lfloor n/kC_{\ref{lem:biggroupspecialfamily}}\rfloor$, we have that $\mathcal{F}_i^+$ and $\mathcal{F}_i^-$ are separable at a distance $q+mf$.
\end{enumerate}
\end{lemma}
\begin{proof}
Pick some $z_\mathcal{S}$ between $2$ and $5$ so that $k-4-z_\mathcal{S}$ is positive and divisible by $4$, note that this is possible as $k\geq 10$. Pick any $f\neq 0$. Pick some $s$ so that $q+m\cdot f\neq 0$ for any $m\in \{0,1,2,\ldots, k, k+1\}$ for $q:=-((k-4-z_\mathcal{S})/4)f-s$. Indeed, there are $0.9n$ such choices of $s$, as $k\leq n^{0.001}$.
\begin{claim}\label{sgexists}
We can find $\mathcal{S}_G$ satisfying \ref{si1} and \ref{si2}.
\end{claim}
\begin{proof}
Set $k'=z_\mathcal{S}$, recalling that $k'\geq 2$.  Suppose that we have found a maximal family $\mathcal{S}_G$ satisfying \ref{si1} and \ref{si2} and suppose that $|\mathcal{S}_G|<n/(kC)$ for some $C$. We will derive a contradiction for $C$ sufficiently large.

\par Let $\bar{G}$ be the collection of rainbow path candidate $k'$-tuples $(g_1,\ldots, g_{k'})$ with $g_1+\cdots+g_{k'}=s$ which are disjoint with $\bigcup \mathcal{S}_G$ (observing this set has size at most $5n/(kC)$ by assumption). $|\bar{G}|\geq n^{k'-1}/4$ by Lemma~\ref{Lemma_counting_good_tuples} (supposing $C\geq 100$). If for some $\bar{t}\in \bar{G}$ we have that $\mathcal{S}_G^*$ becomes non-dissociable upon the addition of $\bar{t}$, there must exist some $\bar{t}'\in \mathcal{S}_G$ $j,j'\in[k'-1]$ such that $\sum_{i\in [j]}\bar{t}'_i=\sum_{i\in [j']}\bar{t}_i$. There are at most $(n/kC)k'n^{k'-2}\leq n^{k'-1}/C$ such $\bar{t}$, meaning that there is a $\bar{t}\in \bar{G}$ that we can add to $\mathcal{S}_G$ without breaking \ref{si1} and \ref{si2}, a contradiction.
\end{proof}

It remains to construct $\mathcal{F}_G$. Suppose $\mathcal{F}_G$ is a family of $4$-tuples satisfying the properties with size at most $n/(kC) - 1$ (where $C$ is a sufficiently large constant). We will show that $\mathcal{F}_G$ can be extended.
\par Fix some $f^+,f^- \in G\setminus\{0\}$ such that $f^+ + f^- = f$, and the following two properties hold.
\begin{enumerate}
    \item For any $\mathcal{F}_i\in \mathcal{F}_G$ we have that $f^+,f^-\neq \sum \mathcal{T}_i$ for any $\mathcal{T}_i\subseteq \mathcal{F}_i$.
    \item $f^+ + f^-\neq \pm q+mf$ for any $m\in \{0,1,\ldots, k, k+1\}$.
\end{enumerate}
Such $f^+,f^-$ with the first property exist as long as $C>50$, as $\sum \mathcal{T}_i^{\pm}$ can take at most $20n/C$ distinct values due to the assumption on the size of $\mathcal{F}_G$. Such $f^+,f^-$ automatically satisfy the second property as $q+m\cdot f\neq 0$ for any $m\in \{0,1,\ldots k\}$.
\par Let $F^+$ denote the set of ordered triples with sum $f^+$ and $F^-$ denote the set of ordered triples with sum $f^-$, noting $|F^+|=|F^-|=n$.
\begin{claim}\label{fgexists}
Suppose we delete all triples from $F\in F^+$ such that the collection $\mathcal{F}_+\cup \{F\}$ fails to be dissociable. This deletes at most $10n/C$ triples.
\end{claim}
\begin{proof}
If for some $F\in F^+$ we have that $\{\mathcal{F}_i^+\}\cup {F}$ is not dissociable, there must exist some $F'\in \mathcal{F}^+$ and $j'\in \{1,2\}$ and $j\in\{1,2\}$ such that $\sum_{i\in[j']} F'(i)=\sum_{i\in[j]} F(i)$. It cannot be that $j=2$ by the first property coming from our choice of $f^+$. By the bound on $|\mathcal{F}_G|$, there are at most $10n/C$ distinct values the quantity $\sum_{i\in[j']} F'(i)=:w$ can take. For each such $w$, there is at most one $F\in F^+$ with $F'(1)=w$. This implies that in total there are at most the claimed number of triples which make the corresponding collection not dissociable.
\end{proof}
\begin{claim}\label{midclaim}
Suppose we delete all tuples from $F\in F^+$ such that $F$ and the $1$-tuple $(f^-)$ are not separable at a distance $q+m\cdot f$ for some $m\in\{0,\cdots,k\}$. This deletes at most $2k$ tuples.
\end{claim}
\begin{proof}
If for some $F\in F^+$ we have that $F$ and $(f^-)$ are not separable at a distance $q+m\cdot f$ for some $j\in \{1,2\}$, then we must have $\sum _{i\in[j]}F(i) + f^- = \pm (q+m\cdot f)$ for some $m\in\{0,\ldots, k\}$. Here, $j=2$ is precluded by the second property of $f^+$ and $f^-$. For each of the $2k$ possible values of $\pm(q+m\cdot f)-f^-$, there exists at most one $F\in F^+$ such that $F(1)= \pm(q+m\cdot f)- f^- $, which implies the claim.
\end{proof}
\begin{claim}\label{oldclaim}
There are at least $n/4$ tuples in $F^+$ which are rainbow path-candidates and which contain no coordinate $F(i)$ also present in an element of $\mathcal{S}_G$ or $\mathcal{F}_G$.
\end{claim}
\begin{proof}
This is immediate by Lemma~\ref{Lemma_counting_good_tuples} and bounding $|\bigcup \mathcal{S}_G \cup \bigcup \mathcal{F}_G|$.
\end{proof}
\begin{claim}\label{finalclaim} Deleting all tuples $F\in F^+$ with $F(1)= \pm \sum \mathcal{T}_i$ or $F(2)= \pm \sum \mathcal{T}_i$  for some $\mathcal{T}_i\subseteq \mathcal{F}_i\in \mathcal{F}_G$, we delete at most $80n/C$ elements.
\end{claim}
\begin{proof} There are at most $20n/C$ possible values for the quantity $\pm \sum \mathcal{T}_i$ by the upper bound on the size of $\mathcal{F}_G$. This implies the claim, as $F(1)$ (and $F(2)$) is a distinct value for each $F\in F^+$.
\end{proof}
\par By the bounds coming from the claims, we can fix $\mathcal{F}_{new}^+$ to be a $2$-tuple from $F^+$ which is a rainbow path candidate disjoint with the earlier sets, keeps $\mathcal{F}^+$ dissociable, and is separable with $(f^-)$ at a distance $q+m\cdot f$ for each $m\leq k$.
\par Now, we perform the analogous steps for $F^-$. Claim~\ref{fgexists} and \ref{oldclaim} (thinking of $F^+_{new}$ as an element of $\mathcal{F}_G$ to ensure disjointness) also hold when $+$ is replaced by $-$, giving us at least $n/5$ potential elements of $F^-$ we can select while maintaining dissociability of $\mathcal{F}^-$, disjointness with previous tuples, and rainbow path candidacy.
\par In addition, we delete the elements of $F^-$ which are not separable with $\mathcal{F}^+_{new}$ at a distance $q+mf$ for some $m\in \{0,1,\ldots, k\}$. For any $F\in F^-$, it is already impossible for $\sum_{i\in [j]} \mathcal{F}^+_{new}(i) +\sum_{i\in [j']} F(i)=\pm(q+m\cdot f)$ when $j'=2$ (by the property from Claim~\ref{midclaim}). When $j'=1$, note that $ \pm(q+m\cdot f) - \sum_{i\in [j]} \mathcal{F}^+_{new}(i)$ can take at most $4k$ distinct values $v$, and we only need to delete at the at most $4k$ many $F\in F^-$ with $F(i)=v$.
\par For each $F\in F^-$, consider the $4$-tuple $\mathcal{F}_{new}=(\mathcal{F}^+_{new}(1), \mathcal{F}^+_{new}(2), F(1), F(2))$, and note that this is always a rainbow sequence. If $F\in F^-$ makes $\mathcal{F}_G\cup \{\mathcal{F}_{new}\}$ not near-dissociable, we delete $F$ from $F^-$. To count how many such $F$ there are, suppose that for some $\mathcal{F}\in \mathcal{F}_G$, we have that $\sum_{i\in[j]}\mathcal{F}(i)= \sum_{i\in[j']}\mathcal{F}_{new}(i)$ where $j,j'\in[3]$. It is impossible that $j'\in \{1,2\}$ due to Claim~\ref{finalclaim} and the first property of $f^+$. Note there are at most $20n/C$ potential values of $\sum_{i\in[j]}\mathcal{F}(i)$ due to the bound on the size of $\mathcal{F}_G$. This implies that for the relevant equality to hold, $F(1)$ needs to belong to a set of size $20n/C$, so in this step we delete at most $20n/C$ elements from $F^-$. Similarly, if $(\mathcal{F}^+_{new}(1), \mathcal{F}^+_{new}(2), F(1), F(2))$ is not a path candidate, it must be that a partial sum of the sequence is $0$. This partial sum cannot contain both of $F(1)$ and $F(2)$, as the whole sum is $f\neq 0$, and $\mathcal{F}^+_{new}(2)\neq -F(1)-F(2)$ by Claim~\ref{finalclaim}, and $(F(1),F(2))$ is a path candidate. But the partial sum has to contain $F(1)$, as $(\mathcal{F}^+_{new}(1), \mathcal{F}^+_{new}(2))$ alone gives a path candidate. This means that at most $2$ extra values of $F(1)$ are forbidden if $(\mathcal{F}^+_{new}(1), \mathcal{F}^+_{new}(2), F(1), F(2))$ is to be a rainbow path candidate.

Selecting $C$ large, we can fix a value of $F\in F^-$ so that setting $\mathcal{F}^-_{new}=F$, we successfully extend $\mathcal{F}_G$, as desired.
\end{proof}

\section{Nibble with some determinism}\label{sec:nibble}

In this section we give a proof of Lemma~\ref{lem:deterministicnibble}.

\begin{observation}\label{obs:typical} Let $\mathcal{E}$ be an equation of the form $\pm a \pm b \pm c = 0$.
    Let $H$ be a tripartite hypergraph obtained by taking three copies of some group $G$ of order $n$, and letting $(a,b,c)\in G^3$ be an edge whenever it is a solution to $\mathcal{E}$. Then, $H$ is $(0, 1,n)$-typical.
\end{observation}
\begin{proof}
    For a proof for when $\mathcal{E}$ is $a+b+c=0$, see Observation 3.3 in \cite{muyesser2022random}. For other equations of this form, the proof is essentially identical.
\end{proof}
Typical graphs have the following useful pseudorandomness property.
\begin{lemma}[\cite{muyesser2022random}]\label{Lemma_one_random_set_nearlyish_regular}
Let $H=(A,B,C)$ be a tripartite linear hypergraph that is $(0, 1,n)$-typical. Let $p\geq n^{-1/600}$ and let $A'\subseteq A$ be $p$-random.
Then, with probability at least $1-1/n^3$,
the following holds. For any $B'\subseteq B$, there are at most $n^{9/10}$ vertices $c\in C$ with $e_{H}(A',B', c)\neq p |B'|\pm  n^{9/10}$.
\end{lemma}

\begin{lemma}\label{Lemma_one_random_set_nearly_regular}
Let $p\geq n^{-1/600}$ and let $X\subseteq G$ be $p$-random. Then, with probability at least $1-8/n^{3}$, the following holds.
For any $Y\subseteq G$, for all but at most $8n^{9/10}$ vertices $g\in G$, and for each equation of the form $\mathcal{E}:=\pm g \pm x \pm y = 0$ (where $g$ is a constant and $x$ and $y$ are free variables), we have that there are $p|Y|\pm  n^{9/10}$ many $(x,y)\in X\times Y$ such that $x,y$ and $g$ satisfy $\mathcal{E}$.
\end{lemma}
\begin{proof}
    Thanks to Observation~\ref{obs:typical}, we can apply Lemma~\ref{Lemma_one_random_set_nearlyish_regular} to the corresponding hypergraph defined by each of the $2^3=8$ possible equations $\mathcal{E}$, and with probability at least $1-8/n^{3}$, we ensure that the conclusion of Lemma~\ref{Lemma_one_random_set_nearlyish_regular} holds for each of these hypergraphs. The desired statement follows immediately.
\end{proof}

\par We say that $g\in G$ is \textbf{generic} if $g\neq \id$ and there are at most $n^{1/2}$ solutions to $x^2=g$ in $G$. Let $N(G)$ denote the set of non-generic elements and note that $|N(G)|\leq n^{1/2}$.

\begin{observation}\label{obs:nongenericobs}
Let $G$ be an abelian group of order $n$ and let $A\subseteq G$ be a multiset of order $k$. Consider the sets $A+g$ for each $g\in G$. Then, at most $kn^{-1/10}$ many such sets have more than $n^{3/5}$ many non-generic elements. Also, there are at most $kn^{-1/10}$ sets $A-g$ with more than $n^{3/5}$ many non-generic elements.
\end{observation}
\begin{proof}
There are $\leq k n^{1/2}$ tuples $(a,g)\in A\times G$ where $a+g$ is non-generic. Let $\#$ be the number of $g\in G$ such that there are $\geq n^{3/5}$ many $a\in A$ such that $a+g$ is non-generic. Then, $\#\cdot n^{3/5}\leq k n^{1/2}$, so $\#\leq kn^{-1/10}$. The same argument applies when $+$ is replaced by $-$.
\end{proof}

Recall that given a fixed graph $F$, a \textbf{packing} of $F$ in some other graph $G$ is just a collection of vertex-disjoint copies of $F$ in $G$. When we talk about rainbow packings, we always mean that there is no colour repetition in edges across all copies of $F$ in the packing.
\begin{lemma}\label{lem:exhaustingdeterministic} There exists an absolute constant $\eps_{\ref{lem:exhaustingdeterministic}}>0$ such that the following holds. Let $p\geq n^{1-\eps_{\ref{lem:exhaustingdeterministic}}}$. Let $G$ be a group of order $n$. Let $V_2,V_3\subseteq G$ be disjoint $p$-random, let $C_1,C_2, C_3\subseteq G$ be disjoint $p$-random, sampled independently with $V_2,V_3$. The following holds with probability at least $1-1/n^{2.9}$.
\par Let $V_1, V_4\subseteq G\setminus (V_2\cup V_3)$ with $|V_1|=|V_2|=(p\pm n^{-0.1})n$. Let $f\colon V_1\to V_4$ be a bijection. Then, $\vec{K}_G[V_1, V_2, V_3; C]$ contains a rainbow packing of at least $n^{1-1/10^{5}}$ paths of length $3$, directed $V_1\to V_2\to V_3\to V_4$, such that for all paths $\vec{P}$ in the packing and $v_1\in V_1\cap V(\vec{P})$, we have that $f(v_1)\in V_4\cap V(\vec{P})$.
\end{lemma}
\begin{proof}

Each of the following holds with probability at least $1-O(1/n^{3})$, thus they all simultaneously hold with probability at least $1-1/n^{2.9}$.

\begin{enumerate}[label = {(\arabic{enumi}})]
    \item Lemma~\ref{Lemma_one_random_set_nearly_regular} holds for $X$ set to be each of $V_2$, $V_3$, $C_1$, $C_2$ and $C_3$.
    \item For each $i$, $|V_i|,|C_i|=(p\pm n^{-0.1})n$, by Chernoff's bound.
    \item For every colour $c\in G\setminus \{0\}$ and vertex pair $v,w\in G$ such that $v-w-c$ is generic, we have that there exists $(p^4\pm n^{-0.1})n$ many rainbow paths of length $3$ directed $v\to V_2\to V_3 \to w$ with edge colours $(c_1,c,c_3)$ for some $(c_1,c_3)\in C_1\times C_3$. If $v-w-c$ is not generic, we have that there exists at most $(p^4+ n^{-0.1})n$ such paths.
    \begin{proof}
    Consider all tuples $(v,v_2,v_3,w, c_1,c,c_3)$ where $v_2,v_3,c_1,c_3\in G$, and $v-v_2=c_1$, $v_2-v_3=c$ and $v_3-w=c_3$. There are $n$ such tuples. Note $c_1+c_3=v-v_2+v_3-w=v-w-c$ which is generic. This means for all but at most $n^{1/2}$ tuples, $c_1\neq c_3$. As $c\neq 0$, for all tuples $v_2\neq v_3$. For each of the $n-n^{1/2}$ tuples where $c_1\neq c_3$, the probability of $(v_2,v_3,c_1,c_3)\in V_2\times V_3\times C_1\times C_3$ is $p^4$. Letting $X$ denote the expected number of paths of the desired form, we obtain that $\mathbb{E}[X]=(p^4\pm n^{-1/2})n$. Further, $X$ is $2$-Lipschitz, so the desired concentration follows from Azuma's inequality. When $v-w-c$ is not generic, the same argument applies except we only have an upper bound on $\mathbb{E}[X]$.
    \end{proof}
    \item For every pair of vertices $v,w\in G$ such that $v-w$ is generic,  we have that there exists $ (p^3\pm n^{-0.05})n$ many rainbow paths directed $v \to V_3 \to w $ with edge colours from $ C_2\times C_3$. If $v-w$ is not generic, we have that there exists at most $(p^3+ n^{-0.1})n$ such paths.
    \item For every pair of vertices $v,w\in G$ such that $v-w$ is generic,  we have that there exists $ (p^3\pm n^{-0.05})n$ many rainbow paths directed $v \to V_2 \to w $ with edge colours from $ C_1\times C_2$. If $v-w$ is not generic, we have that there exists at most $(p^3+ n^{-0.1})n$ such paths.
\end{enumerate}
\par The proofs for (4) and (5) are essentially identical to the proof for (3), hence we omit them.
\par Now, suppose $V_i$ and $C_i$ all of the properties, and fix a bijection $f\colon V_1\to V_4$.  Let $\mathcal{H}$ be the hypergraph consisting of edges $(v_1,v_2,v_3, v_4,c_1,c_2,c_3)\in V_1\times V_2\times V_3\times V_4\times  C_1\times C_2\times C_3$ where $v_1-v_2=c_1$, $v_2-v_3=c_2$, $v_3-v_4=c_3$ and $f(v_1)=v_4$. Our goal is to find a matching covering all but $n^{1-1/1000}$ vertices in this hypergraph. We sometimes refer to the edges of this hypergraph as paths. We will show that there is a set $S$ of $\leq n^{1-1/100}$ vertices we can delete from $\mathcal{H}$ so that the resulting hypergraph $\mathcal{H}'$ is almost regular, that is, for all $v\in V(\mathcal{H})$, $d(v)=(p^5\pm n^{-0.05})n^2$.
\par Towards that goal, set $S$ to include
\begin{itemize}
    \item the $\leq 90n^{9/10}$ vertices of $G$ coming from Lemma~\ref{Lemma_one_random_set_nearly_regular} applied with each of $V_2,V_3,C_1,C_2,C_3$
    \item the $\leq 2|V_1|n^{-1/10}\leq 2n^{9/10}$ elements of $G$ coming from Observation~\ref{obs:nongenericobs} applied with the multiset $\{v-f(v)\colon v\in V_1\}$ and both $+$ and $-$
\end{itemize} so we have $|S|\leq 100n^{9/10}$. We will show all vertices of $\mathcal{H}$ not in $S$ have degree $(p^5\pm n^{-0.05})n^2$. Since $S$ is small, this shows that $S$ has the desired property. We consider several cases.
\par Let $v_1\in V_1\setminus S$, and set $v_4=f(v_1)\in V_4\setminus S$. For all but $n^{1/2}$ many $c_2\in C_2$ we have that $v-w-c_2$ is generic. For such $c_2$, we have by $3.$ that there are $(p^4+n^{-0.1})n$ paths passing through both $v_1$ and $c_2$. Combined with the bound on the size of $C_2$ coming from (2), this shows the desired upper and lower bound on $d(v_1)$ because through the few $c_2$ such that $v-w-c_2$ is non-generic, there exists at most $10n$ paths passing through both $v_1$ and $c_2$, giving in total $O(n^{3/2})$ such paths.
\par If $v_4\in V_4\setminus S$, set $v_1=f^{-1}(v_4)$ and apply the result from the previous paragraph.
\par Let $c_1\in C_1\setminus S$. From Lemma~\ref{Lemma_one_random_set_nearly_regular}, we have that there exists $p^2n \pm n^{9/10}$ directed $c_1$ coloured edges from $V_1$ to $V_2$. As $c_1\notin S$, for all but $n^{3/5}$ many $v\in V_1$, $v-w-c_1$ is generic, and so for such $v$, we have that $(v-c_1)-w$ is generic, so we can apply (4) to obtain that there exists $(p^4+n^{-0.1})n$ paths to $f(v)$ passing through $c_1$. Combined with the bound on $|V_1|$, this gives the desired bound on $d(c_1)$, as the number of paths going through the $c_1$ such that $v-w-c_1$ is non-generic is too small to influence the count, as before.
\par Let $c_3\in C_3\setminus S$. This case follows by a symmetric argument with the $c_1\in C_1\setminus S$ case, using (5) in place of (4).
\par Let $c_2\in C_2\setminus S$. Let $(v,w=f(v))\in V_1\times V_2$ and suppose that $v-w-c_2$ is generic. Then by (3) there are $(p^4+n^{-0.1})n$ paths passing through $v$, $w$, and $c_2$. As $c_2\notin S$, we have that all but $n^{3/5}$ values of $v\in V_1$, $v-w-c_2$ is generic. This, with the bound on $|V_1|$ implies the desired bound on $d(c_2)$, again because there are few paths passing through $c_2$ with $v-w-c_2$ non-generic.
\par So $S$ has the desired properties, making $\mathcal{H}'$ almost-regular. Let $\mathcal{H}''$ be the hypergraph obtained by contracting $v$ and $f(v)$ to a single vertex for each $v_1\in V_1$. Note that as $f$ is a bijection, and any edge through $v$ has to pass through $f(v)$ as well, this does not change the regularity parameters of any of the other vertices in $\mathcal{H}'$. To see that this satisfies the hypotheses of Corollary~\ref{cor:nibble}(1), the only thing left to check is the co-degree condition. This is equivalent to obtaining an upper bound on the number of tuples $(v_1,v_2,v_3,v_4,c_1,c_2,c_3)\in \mathcal{H}$ where the values of $2$ coordinates are fixed, and it is not the case that these two coordinates are the first and the fourth (since the corresponding vertices have been contracted). This means that we are counting solutions to a system of equations with $7$ free variables and $6$ independent constraints, hence there are at most $n$ such solutions. This gives the desired co-degree bound. Corollary~\ref{cor:nibble}(1) then gives the desired result.
\end{proof}

\begin{proof}[Proof of Lemma~\ref{lem:deterministicnibble}]
\par Let $q\leq \eps_{\ref{lem:deterministicnibble}} p^3/k^{100}$ be a rational number with denominator at most $n$, observing that there are at most $n$ values of such $q$. First, we will show that with probability at least $1/n^2$, the statement holds for any $V_D,C_D\subseteq G$ with $|V_D|=|C_D|=qn$.
\par Suppose first that $k=3$. Let $r=(p-2q)/3$ and let $R_1^{(1)},R_1^{(2)}, R_1^{(3)}$ be disjoint $r$, $r$ and $r+2q$ random (respectively) sets partitioning $R_1$. Let $R_2^{(1)},R_2^{(2)}, R_2^{(3)}, R_2^{(4)}, R_2^{(5)}$ be disjoint $q$, $q$, $r$, $r$ and $r$-random (respectively) sets partitioning $R_2$. With probability at least $1-O(1/n^{2.9})$ (we assume here that $q\geq n^{-1/100}$, otherwise the argument up to finding $M_1$ can be discarded, and the $M_2$ found at the end of the argument satisfies the requirements), Lemma~\ref{Lemma_one_random_set_nearly_regular} holds for $R_2^{(1)}$ and $R_2^{(2)}$ and Lemma~\ref{lem:exhaustingdeterministic} holds with $(V_2,V_3)=(R_1^{(1)}, R_1^{(2)})$ and $(C_1,C_2,C_3)=(R_2^{(3)}, R_2^{(4)}, R_2^{(5)})$. Also, by Chernoff's bound the following holds for all cycle-candidate triples $(a,b,c)$ simultaneously with probability at least $1-1/n^{10}$: there exists at least $p^3n/1000$ vertex-disjoint $3$-cycles in $\vec{K}_G[R_1^{(3)}]$ with colour sequence $(a,b,c)$ (this holds with high probability by Lemma~\ref{lem:patternfinding} as well, indeed see Lemma~\ref{lem:pathcyclecandidates}, but here we cite Chernoff's bound directly to obtain an explicit bound on the probability). Finally, with probability at least $1-1/n^{10}$, all random sets are at most $n^{0.6}$ elements away from their expectations. With probability at least $1-1/n^2$ all of these properties hold simultaneously.
\par Now let $V_D,C_D$ be given. Let $H_G[R_2^{(1)}, R_2^{(2)}, C_D]$ denote the $3$-partite $3$-uniform hypergraph on the indicated parts where triples are edges if and only if they are zero-sum. From Lemma~\ref{Lemma_one_random_set_nearly_regular} applied with the equation $g+x+y=0$, we have that all but $n^{99/100}$ vertices of $H_G[R_2^{(1)}, R_2^{(2)}, C_D]$ do not satisfy the regularity hypothesis from Corollary~\ref{cor:nibble}(2). Deleting such vertices, we obtain a $(n^{-0.01},q^2,qn)$-regular linear tripartite hypergraph, so Corollary~\ref{cor:nibble}(2) implies that all but $n^{1-1/700}$ elements of $R_2^{(1)}\cup R_2^{(2)}\cup C_D$ can be covered by disjoint zero-sum triples, denote these triples by $\mathcal{T}$. If necessary, delete at most one edge from $\mathcal{T}$ so that $0$ is not used on any triple, meaning that the remaining triples can be ordered to be cycle-candidates (see, for example, Lemma~\ref{alspach}). Using the property of $R_1^{(3)}$ repeatedly for each triple in $\mathcal{T}$, we can find a matching $M_1$ saturating all triples in $\mathcal{T}$ (and nothing else) in $\mathcal{H}_k[R_1^{(3)}; \bigcup \mathcal{T}]$. Now, invoke Lemma~\ref{lem:exhaustingdeterministic} with $V_1=V_4=(R_1^{(3)}\setminus V(M_1))\cup V_D$ (noting $|V_1|=|R_1^{(3)}|$) and $f$ set to be the identity function. This gives that $\mathcal{H}_k[(R_1^{(1)}\cup R_1^{(2)}\cup R_1^{(3)}\setminus V(M_1))\cup V_D; R_2^{(3)}, R_2^{(4)}, R_2^{(5)}]$ has a matching covering all but $10n^{1-1/10^5}$ vertices, say $M_2$. Then, $M_1\cup M_2$ is the desired matching.
\par Union bounding over all potential values of $q$, we obtain that with high probability, for each $q$ the statement holds. This is sufficient to deduce the assertion, as $qn$ is always an integer.

\par Now, suppose that $k\geq 4$. For each $i\in [k]$, and $j\in\{1,2\}$ let $R_j^{(i)}$ be a $((p+q)/k)$-random set for $i\geq 2$ and $((p+q)/k - q)$-random set for $i=1$, partitioning $R_j$. Similarly to the $k=3$ case, we will fix some rational $q\leq pn/k^{100}$ with denominator at most $n$, and prove that the desired statement holds with probability at least $1-1/n^{1.1}$. \par With probability at least $1-1/n^{1.3}$, Lemma~\ref{lem:directedwithdeterministic} holds for $(A,B,C)=(R_1^{(i)}, R_1^{(i+1)}, R_2^{(i)})$ with $\ell=((p+q)/k)n$ for each $i\in [k-3]$ (using that $k$ is polylogarithmic in $n$ for the union bound). With probability at least $1-1/n^{10}$, Lemma~\ref{lem:exhaustingdeterministic} holds with $(V_1,V_2)=(R_1^{(k-1)}, R_1^{(k)})$, $(C_1,C_2,C_3)=(R_2^{(k-2)}, R_2^{(k-1)}, R_2^{(k)})$. With probability at least $1-1/n^2$, each random set is within $n^{0.6}$ elements of its expected size. With probability at least $1-1/n^{1.1}$, all these properties hold simultaneously.
\par Let $V_D, C_D$ be given. Apply Lemma~\ref{lem:directedwithdeterministic} with random sets $(R_1^{(1)}, R_1^{(2)}, R_2^{(1)})$ and $(A',B',C')=(R_1^{(1)}\cup V_D, R_1^{(2)}, R_2^{(1)}\cup C_D)$ to find a matching that saturates all but $n^{1-1/10^7}$ vertices. Continue invoking Lemma~\ref{lem:directedwithdeterministic} with corresponding random sets and $(A',B',C')=(R_1^{(i)}, R_1^{(i+1)}, R_2^{(i)})$ for each $i\in[k-3]\setminus\{1\}$. In both of these applications, we may delete/add $O(n^{0.78})$ elements from the corresponding sets so that they have size precisely $\ell$ or $\ell+\lfloor n^{1-10^{-5}}\rfloor$ (depending on whether they are vertex or colours sets), so that the hypotheses of Lemma~\ref{lem:directedwithdeterministic} are satisfied, and then if necessary we can delete all edges passing through a dummy vertices/colours. Deleting all vertices that fail to be covered by one of the $k-3$ matchings found via the previous applications of Lemma~\ref{lem:directedwithdeterministic}, we delete at most $kn^{1-1/10^7}\leq n^{1-1/(2\cdot10^7)}$ vertices. The remaining vertices form directed paths following sets $R_1^{(1)}\to R_1^{(2)}\to\cdots\to R_1^{(k-2)}$. Let $V_1\subseteq R_1^{(1)}$ and $V_4\subseteq R_1^{(k-2)}$ be the vertices used by these directed paths, noting $|V_1|=|V_4|$, and let $f\colon V_1\to V_4$ be the bijection induced by the two endpoints of each directed path. Now, Lemma~\ref{lem:exhaustingdeterministic} allows us to complete all but $n^{1-1/10^5}$ of these paths of length $k-3$ into a $k$-cycle using the remaining random sets, which gives the desired matching in $\mathcal{H}_k$. Union bounding over the potential values of $q$, we obtain the desired result, as in the $k=3$ case.
\end{proof}

\section{Zero-sum absorption}\label{sec:zerosumabsorption}
In this section we prove Lemma~\ref{lem:zerosumabsorption}. Throughout this section, whenever a constant $C$ appears inside the statement of a lemma, this should be read as ``there is a sufficiently large absolute constant $C$ so that the statement holds with this value of $C$''.

\subsection{Cover-down step: saturating vertices and colours}\label{sec:saturating}

\subsubsection{Covering vertices}

The next lemma gives us a way to find edges of $\mathcal{H}_k$ that pass through a specific set of vertices.

\begin{lemma}\label{lem:exhaustingvertices}
Let $p\geq n^{-1/700}$. Let $3 \leq k\leq \log^{10} n$. Let $R_1, R_2$ be $p$-random subsets of $G$ sampled independently. With high probability, the following holds.
\begin{enumerate}
    \item Let $U\subseteq G$ be a set with $|U|\leq p^{300}n/C_{\ref{lem:exhaustingvertices}}$ (recall the convention set in the beginning of Section~\ref{sec:zerosumabsorption}). Let $u,v\in G$ be two vertices, not necessarily distinct. Let $k'$ be such that $2\leq k'\leq k$. Then, if $u\neq v$, there exists a rainbow path of length $k'$ directed from $u$ to $v$ in $\vec{K}_G[(R_1\setminus U)\cup \{u,v\}; R_2\setminus U]$. If $u=v$ and $k'\geq 3$, there exists a directed rainbow cycle of length $k'$ using the vertex $v$.
    \item Let $U\subseteq G$ be a set with $|U|\leq p^{300}n/4C_{\ref{lem:exhaustingvertices}}$. Let $V\subseteq G$ be a set of vertices with $|V|\leq p^{300}n/(4kC_{\ref{lem:exhaustingvertices}})$. Then, $\mathcal{H}_k[(R_1\setminus U)\cup V; R_2\setminus U]$ has a matching saturating $V$ where each matched edge uses exactly one vertex from $V$.
\end{enumerate}
\end{lemma}
\begin{proof} Fix a large constant $C\geq 10^7$ and fix $C_{\ref{lem:exhaustingvertices}}\gg C$. Fix some distinct vertices of $\vec{K}_G$, $u$ and $v$. Fix a set of $n/10$ triples $(c_1,x, c_2)$ where $u\to x\to v$ is a rainbow path of length $2$ with edge sequence $(c_1,c_2)$, and the resulting collection of $x$ and $\{c_1,c_2\}$ are both pairwise disjoint. Such a collection exists because there are at least $n/4$ disjoint $(c_1,c_2)$ with $c_1\neq c_2$ and $c_1+c_2=u-v\neq 0$ and $G$ is an abelian group. By Chernoff's bound, with exponentially high probability, $\vec{K}_G[R_1  \cup \{u,v\}; R_2]$ contains at least $p^3n/10$ such paths. By a union bound over all distinct $u$ and $v$, we have that $\vec{K}_G[R_1  \cup \{u,v\}; R_2]$ contains at least $p^3n/10$ such paths for any choice of $u$ and $v$ with high probability. Call this property $(\ast)$. Also, Lemma~\ref{lem:patternfinding} holds with high probabilty.
\par We claim that a stronger version of part $(1)$ holds when $k'\in\{2,3\}$, with $C_{\ref{lem:exhaustingvertices}}$ replaced with $C$. For $k'=2$, this already follows from the property $(\ast)$ as each element of $U$ (other than $u$ or $v$) can eliminate at most $1$ path. We claim the $k'=3$ case follows from an application of Lemma~\ref{lem:patternfinding}. In the case that $u\neq v$, we can see this by defining a pattern that is a directed path of length $3$, first vertex labelled $u$, last vertex labelled $v$, and colour sequence ($c_1$, $c_2$, $u-v-c_1-c_2$) (labelled in order of proximity to $u$) where $c_1,c_2$ are free variables (and $u$ and $v$ are constants). Note this implies that the vertex sequence is $(u, u-c_1, u-c_1-c_2, v)$ (so that the third property in the definition of a pattern holds). This is a well-defined pattern as each vertex gets a distinct label. Furthermore, the pattern is well-distributed. The colours are separable by \ref{linear}, $u$ and $v$ are separable by \ref{gnotzero} (as $u\neq v$) and the rest of the vertex pairs are separable by \ref{linear}. Also, each label is either a constant, or linear in $c_1$ or $c_2$. Then, Lemma~\ref{lem:patternfinding} gives us a copy of this pattern, which corresponds to the desired rainbow path thanks to Observation~\ref{patternsaresubgraphs}. In the case that $u=v$, we can proceed similarly, this time using a pattern that is a directed cycle of length $3$, with colour sequence $(c_1,c_2,-c_1-c_2)$ and vertex sequence $(u,u-c_1, u-c_1-c_2, u)$, where $c_1,c_2$ are free variables.
\par For larger $k'$, $(1)$ follows from repeated applications of the cases of $k'\in\{2,3\}$. Indeed, any directed path/cycle of length $\geq 4$ can be broken up into directed paths of length $2$ and $3$. While iteratively invoking $(1)$ with $k'\in\{2,3\}$, we extend $U$ at each step, in total adding at most $k$ new elements to $U$. As we know $(1)$ holds with $k'\in\{2,3\}$ with the smaller constant $C$, this reduction is valid, as $p^{900} n/C_{\ref{lem:exhaustingvertices}} + k<2p^{900}n/C_{\ref{lem:exhaustingvertices}}<p^{900} n/C$ (using that $k$ is small and that $C\ll C_{\ref{lem:exhaustingvertices}}$, respectively).
\par Now, let a $U$ be given as in the part $(2)$ of the statement, and include in $U$ all vertices of $V$ (without relabelling). Given a vertex $z\in V$, we can invoke part $(1)$ with $z=u=v$ to find a cycle of length $k$ using $z$. For the next iterations, we apply $(1)$ adding to $U$ the vertices we've used so far, which adds to $U$ at most $k|V|\leq p^{300}n/(4C_{\ref{lem:exhaustingvertices}})$ elements. This means that $U$ never exceeds a size of $p^{300}n/C_{\ref{lem:exhaustingvertices}}$ in any of the iterations, making the applications of $(1)$ valid.
\end{proof}

\subsubsection{Covering colours: Small $k$}
\begin{lemma}\label{lem:pathcyclecandidates}
Let $p\geq n^{-1/700}$. Let $k\leq 50$. Let $R_1$ be a $p$-random subset of $G$. With high probability, the following holds.
\par Let $S$ be a $k$-tuple which is a rainbow path candidate. Let $U\subseteq G$ with $|U|\leq p^{350}n/C_{\ref{lem:pathcyclecandidates}}$. Then, $\vec{K}_G[R_1\setminus U]$ contains a path with colour sequence $S$. Similarly, if $S$ is a cycle candidate, $\vec{K}_G[R_1\setminus U]$ contains a cycle with colour sequence $S$.
\end{lemma}
\begin{proof} With high probability, Lemma~\ref{lem:patternfinding} holds with $R_1$ (and $R_2$ set to be a $p$-random set independent with $R_1$ -- $R_2$ will not be relevant in the proof). Consider a pattern (as in Definition~\ref{def:pattern}) consisting of a directed path on $k+1$ vertices, and label $i$th edge with $c_i$ and label the first vertex of the path with $v$ where $v$ is a free variable. This is enough information to determine the label of the remaining $k$ vertices: the label of the $i$th vertex on the path for $2\leq i\leq k+1$ has to be $v-\sum_{1\leq j<i}c_i$ for the pattern to be well-defined (note also that each vertex gets a distinct label).
\par This pattern is well-distributed. For pairs of colours, separability follows by \ref{gnotzero} as $c_i\neq c_j$ when $i\neq j$. For pairs of vertices, separability follows from \ref{gnotzero} once again, this time using that $S$ is a path-candidate (recall being a path-candidate implies that all partial sums of $S$ are non-zero). Also, all vertices of this pattern are linear in at least one variable, namely, $v$. Also, all colours are constants. Therefore, by Lemma~\ref{lem:patternfinding}, there exists a copy of this pattern in $\vec{K}_G[R_1\setminus U, (R_2\setminus U)\cup S]$, which corresponds to a directed path in $\vec{K}_G$ with $S$ as its colour sequence. To justify this correspondence, recall Observation~\ref{patternsaresubgraphs}.
\par If $S$ was a cycle-candidate instead, an analogous argument works, this time starting with a pattern that is a directed cycle of length $k$, with colour labels coming from $S$, and one of the vertices labelled $v$ which is a free variable. Separability of colours and vertices are by \ref{gnotzero}, using the definition of a cycle-candidate.
\end{proof}
Our cover-down statement for small groups of colours is the following. The proof is omitted, it follows easily by iteratively invoking Lemma~\ref{lem:pathcyclecandidates}.
\begin{lemma}\label{lem:exhaustcolourssmallk}
Let $p\geq n^{-1/700}$. Let $3 \leq k\leq 9$. Let $R_1$ be a $p$-random subset of $G$. With high probability, the following holds.
\par Let $C\subseteq G$ be a set of $k\ell$ colours, admitting a partition into tuples $C_1,\ldots,C_\ell$ where each $C_i$ is a rainbow cycle-candidate. Suppose $\ell\leq p^{400}n/kC_{\ref{lem:exhaustcolourssmallk}}$. Then, $\mathcal{H}_k[R_1\setminus U; C]$ contains a matching of size $\ell$.
\end{lemma}

\subsubsection{Saturating colours: Large $k$}
When $k$ is large, the strategy in the previous section fails for two reasons. Firstly, assuming that the $C_i$ are cycle-candidates as in Lemma~\ref{lem:pathcyclecandidates} would be too much of an ask, as in general we do not have a good way of finding orderings of zero-sum sets in this way (for $k\leq 9$, we get to assume this without loss of generality, relying on Lemma~\ref{alspach}). Even if we were able to find such orderings, there is a second issue which is probabilistic which comes up only when $k\geq  \log n$. The issue is that the expected number of $k$-cycles using a particular colour sequence contained in a $p$-random set is $\leq p^k n$ -- this is too small to have any useful analogue of Lemma~\ref{lem:patternfinding}. Therefore, we pursue a more complicated strategy as follows.
\par Recall that the families $\mathcal{F}_G$ and $\mathcal{S}_G$ were defined in Lemma~\ref{lem:biggroupspecialfamily}.
\begin{lemma}\label{lem:colourpathsseparation}
Let $10 \leq k \leq \log^{10}n$. Let $G$ be a group of order $n$, where $n$ is sufficiently large. The following statements both hold.
\begin{enumerate}
\item Let $v,w\in G$ be distinct vertices with $v-w=q_{G,k} + m \cdot f_{G,k}$ for some natural $m\leq k$. Then, there exists a family $\mathcal{P}$ of size $\geq n/(kC_{\ref{lem:colourpathsseparation}})$ of pairs of rainbow paths $(P,P')$ in $\vec{K}_G$ with the following properties.
\begin{enumerate}
    \item For each $(P,P')\in\mathcal{P}$, $P$ starts on $v$, and $P'$ ends on $w$.
    \item Each path in $\mathcal{P}$ is rainbow, pairwise colour disjoint, and pairwise vertex disjoint except on $\{v,w\}$.
    \item For each $(P,P')\in\mathcal{P}$, $C(P)\cup C(P')\in \mathcal{F}_G$
\end{enumerate}
\item Let $v,w\in G$ be distinct vertices with $v-w=s_{G,k}$. Then, there exists a family of rainbow paths $\mathcal{P}$ of size $\geq n/(kC_{\ref{lem:colourpathsseparation}})$ in $\vec{K}_G$ with the following properties.
\begin{enumerate}
    \item For each $P\in\mathcal{P}$, $P$ starts on $v$, and $P$ ends on $w$.
    \item Each path in $\mathcal{P}$ is rainbow, pairwise colour disjoint, and pairwise vertex disjoint except on $\{v,w\}$.
    \item For each $P\in\mathcal{P}$, $C(P)\in \mathcal{S}_G$
\end{enumerate}
\end{enumerate}
\end{lemma}
\begin{proof} We suppose $C_{\ref{lem:colourpathsseparation}}$ is sufficiently large for the following calculations to go through.
\par Let $\mathcal{P}$ be a maximal family with properties $1(a)$, $1(b)$, and $1(c)$. Suppose $|\mathcal{P}|< n/(kC_{\ref{lem:colourpathsseparation}})$. Supposing that $C_{\ref{lem:colourpathsseparation}}\ll C_{\ref{lem:biggroupspecialfamily}}$, there exists at least $n/(2kC_{\ref{lem:biggroupspecialfamily}})$ many $\mathcal{F}_i\in \mathcal{F}(G)$ which are unused by $\mathcal{P}$. For each such unused $\mathcal{F}_i\in \mathcal{F}(G)$, consider the path $P_i$ with vertex sequence $P_{out}(v, \mathcal{F}_i^+)$ and path $P'_i$ with vertex sequence $P_{in}(w, \mathcal{F}_i^-)$ (these are defined in Section~\ref{sec:grouptheoretictools}). Note that $P_i$ and $P_i'$ are in fact paths as the corresponding colour sequences are path-candidates by Lemma~\ref{lem:biggroupspecialfamily}.
\par Also by Lemma~\ref{lem:biggroupspecialfamily}, the paths $P_i$ and $P'_i$ are vertex disjoint except possibly on $v$ and $w$, as $\mathcal{F}_i^+$ and $\mathcal{F}_i^-$ are separable at a distance $q+mf$. By the dissociability property coming from Lemma~\ref{lem:biggroupspecialfamily}, for each $P_i,P_j$ with $i\neq j$, $P_i$ and $P_j$ are vertex disjoint except on $v$, and similarly $P_i',P_j'$ are vertex-disjoint except on $w$. Using these two properties and assuming that $C_{\ref{lem:colourpathsseparation}}\ll C_{\ref{lem:biggroupspecialfamily}}$ we can find some $i$ for which $\mathcal{F}_i$ is unused in $\mathcal{P}$, $(P_i,P_i')$ is vertex-disjoint with all vertices included in $\mathcal{P}$ (note there are at most $10n/kC_{\ref{lem:colourpathsseparation}}$ such vertices), and $V(P_i)\cap V(P_i')=\emptyset$, contradicting maximality of $\mathcal{P}$.
\par An analogous argument shows the second part of the statement.
\end{proof}

Lemma~\ref{lem:colourpathsseparation} combined with Chernoff's bound implies the following easily.

\begin{lemma}\label{lem:smallcolourpathslargek}
Let $p\geq n^{-1/700}$, let $ 10\leq k\leq \log ^{10} n$. Let $R_1, R_2$ be $p$-random subsets of $G$ sampled independently. With high probability, the following holds.
\par Let $U\subseteq G$ be a set with $|U|\leq p^{50}n/(k C_{\ref{lem:smallcolourpathslargek}})$. Let $v,w\in G$ be distinct vertices.
\begin{enumerate}
    \item  Suppose $v-w= q_{G,k} +  m\cdot f_{G,k}$ for some $0\leq m\leq k$. There exists some $\mathcal{F}\in \mathcal{F}_G$ such that $\vec{K}_G[(R_1\setminus U)\cup\{v,w\}; \mathcal{F}]$ has two vertex disjoint rainbow directed paths of length $2$, one directed away from $v$, one directed into $w$. Furthermore, $\mathcal{F}$ is disjoint with $U$.
    \item Suppose $v-w=s_{G,k}$. There exists some $\mathcal{S}\in \mathcal{S}_{G}$ such that $\vec{K}_G[(R_1\setminus U)\cup\{v,w\}; \mathcal{S}]$ contains a rainbow directed path (of length $z_{\mathcal{S}}$) from $v$ to $w$. Furthermore, $\mathcal{S}$ is disjoint with $U$.
\end{enumerate}
\end{lemma}
The next lemma summarises our cover-down strategy for large $k$. The proof simply iterates parts (1) and (2) of Lemma~\ref{lem:smallcolourpathslargek}, and this works due to properties acquired in Lemma~\ref{lem:biggroupspecialfamily}.
\begin{lemma}\label{lem:exhaustingcolourslargek}
Let $p\geq n^{-1/700}$, let $10 \leq k\leq \log^{10} n$. Let $R_1, R_2$ be $p$-random subsets of $G$ sampled independently. With high probability, the following holds.
\par Let $U\subseteq G$ be a set with $|U|\leq p^{600}n/(kC_{\ref{lem:exhaustingcolourslargek}})$. Suppose $\ell$ is some positive integer with $\ell\leq p^{600}n/(k^2C_{\ref{lem:exhaustingcolourslargek}})$. Let $C\subseteq G$ be a set of $4\ell$ colours admitting a partition into $4$-tuples $C_1,\ldots,C_\ell$ where each $C_i$ is a rainbow path candidate with sum $q_{G,k}$. Then, $\mathcal{H}_k[R_1\setminus U; (R_2\setminus U)\cup C]$ has a matching saturating $C$, and the set of colours $C^*$ used on the matching other than $C$ is closed under $\mathcal{F}_G$ and $\mathcal{S}_G$. Moreover, $C^*$ uses exactly $\ell(k-4-z_\mathcal{S})/4$ tuples from $\mathcal{F}_G$ and exactly $\ell$ tuples from $\mathcal{S}_G$. As a consequence, the matching consists of exactly $\ell$ edges of $\mathcal{H}_k$.
\end{lemma}
\begin{proof} Fix some large $K$ and some $C_{\ref{lem:exhaustingcolourslargek}}\gg K$. With high probability, Lemma~\ref{lem:smallcolourpathslargek} and Lemma~\ref{lem:pathcyclecandidates} both hold.
\par We will first prove the statement when $\ell=1$ with $C_{\ref{lem:exhaustingcolourslargek}}$ replaced with $K$ (note this strengthens the statement). Initialise $U_1=U$, and for each $i\in \{1,\ldots, \ell\}$, do the following.
\par In $\vec{K}_G[R_1\setminus U_i; (R_2\setminus U_i)\cup C]$, find a rainbow path $P_0$ with colours $C_1$ via Lemma~\ref{lem:pathcyclecandidates}. Note that $\text{start}(P_0)-\text{end}(P_0)=q_{G,k}+m\cdot f_{G,k}$ for some natural $m\leq k$ (in fact $m=0$). For a natural $i$, while $k-|P_i|>z_{\mathcal{S}}$, apply Lemma~\ref{lem:smallcolourpathslargek}(1) to extend $P_i$ into another rainbow path $P_{i+1}$ using a set of $4$ extra colours coming from an element of $\mathcal{F}_G$ (disjoint with $U_i$). Note that this preserves that $\text{start}(P_{i+1})-\text{end}(P_{i+1})=q_{G,k}+m\cdot f_{G,k}$ for some $0\leq m\leq k$. Update $U_{i+1}$ to include the new colours and vertices used in $P_{i+1}$.
\par In at most $k$ steps, this procedure yields a path $P$ with $k-|E(P)|=z_{\mathcal{S}}$ due to the divisibility conditions coming from Lemma~\ref{lem:biggroupspecialfamily}, note also $|U_i|\leq |U|+ 10i$, so we add at most $10k$ new elements to $U$ throughout the process. At this point, we can apply Lemma~\ref{lem:smallcolourpathslargek}(2) to complete $P$ into a rainbow cycle, say $\mathcal{C}$, and therefore an edge of $\mathcal{H}_k$. Also, observe that while building $\mathcal{C}$, we used $(k-4-z_{\mathcal{S}})/4$ tuples from $\mathcal{F}_G$ and one tuple from $\mathcal{S}_G$.
\par When $2\leq \ell \leq p^{600}n/(k^2C_{\ref{lem:exhaustingcolourslargek}})$, we can repeat the above procedure for each of $C_2,\ldots, C_\ell$, at each iteration including in $U$ the set of at most $10k$ vertices and colours used in the previous step. This would add at most $10k\ell\leq 10p^{600}n/kC_{\ref{lem:exhaustingcolourslargek}}$ new elements to $U$, which means $U$ will never exceed a size of $11p^{600}n/kC_{\ref{lem:exhaustingcolourslargek}}\ll p^{600}n/kK$ throughout the process. Hence at each step, we can invoke (the stronger version of) the $\ell=1$ case.
\end{proof}

\begin{figure}
    \centering
    \includegraphics[width=1\textwidth]{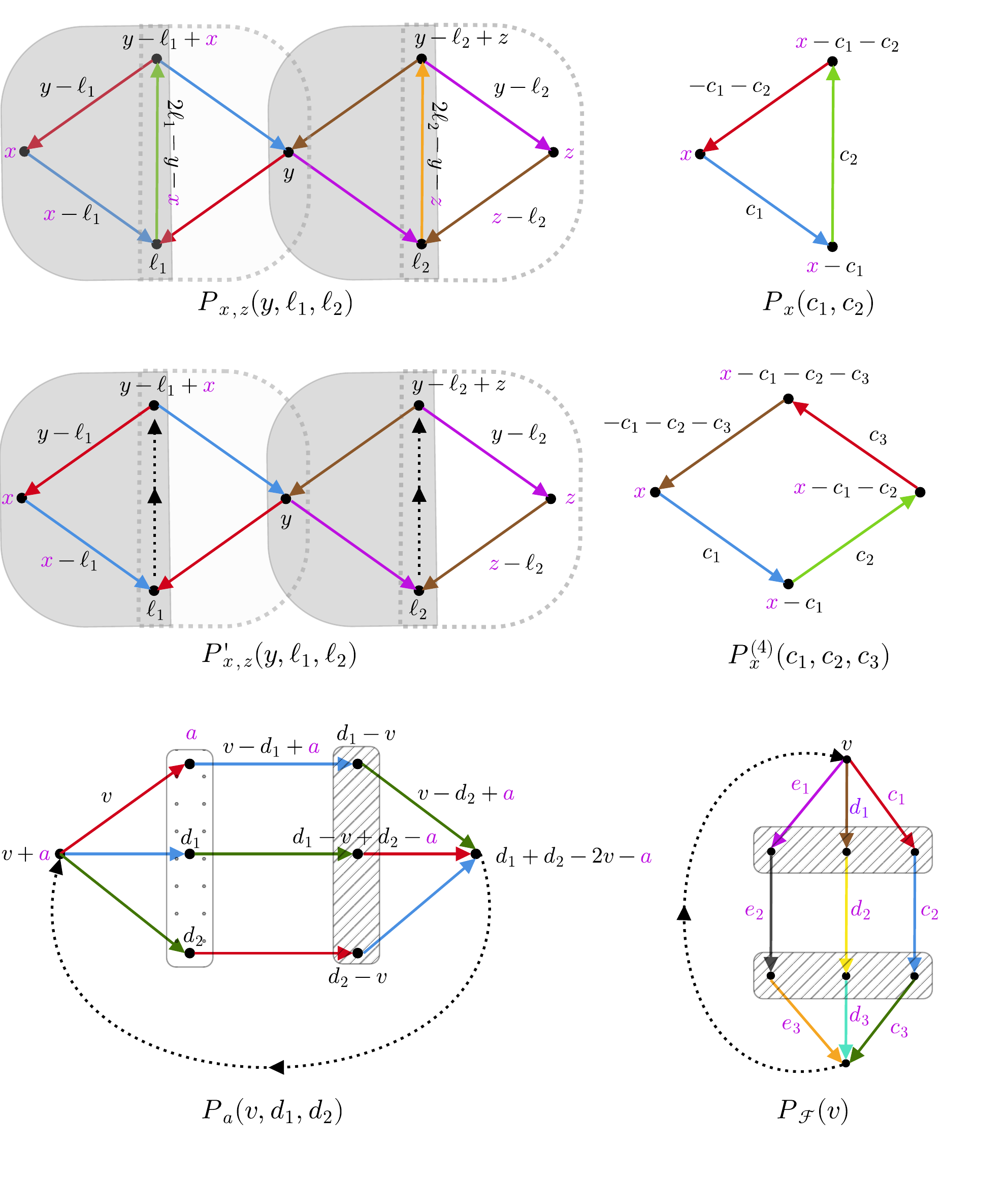}
    \caption{Several patterns (see Definition~\ref{def:pattern}) used in Section~\ref{sec:distributive}. Edge colours correspond to edge labels, so for a collection of edges with the same colour, the label is written for only one of the edges. Free variables are denoted in black letters, and elements of $G$ are denoted in pink letters. The dashed arrows indicate that after a copy of the pattern is found, a rainbow directed path (of the appropriate length) between the indicated vertices will be found. For triples tiled with the diagonal lines, in the proof we find a $2$-absorber for the indicated vertices. }
    \label{fig:gadgets}
\end{figure}

\subsection{Distributive absorption in $\mathcal{H}_k$}\label{sec:distributive}

\begin{definition}
Let $H$ be a hypergraph and let $\mathcal F= \{S_1, \dots, S_t\}$ be a family of subsets of $V(H)$. We say that a set of vertices $R$ $m$\textbf{-absorbs} $\mathcal F$ if for every subfamily   $\mathcal F'\subseteq \mathcal F$  of size $m$, there is a hypergraph matching whose vertex set is exactly $R\cup \bigcup_{S_i\in \mathcal F'} S_i$.
\end{definition}

We will build the desired absorbing structures by finding collections of small subgraphs with certain properties. Each structure found in this section is formed by combining patterns coming from Figure~\ref{fig:gadgets}. Therefore, the following result is crucial, as it will allow us to use Lemma~\ref{lem:patternfinding} numerous times throughout the section (we remark that $P_\mathcal{F}(V)$ is excluded in the statement of the following lemma as we treat this pattern separately).

\begin{lemma}\label{lem:everythingwelldistributed}
In Figure~\ref{fig:gadgets}, all patterns depicted except for $P_\mathcal{F}(V)$ are well-defined patterns which are well-distributed.
\end{lemma}
\begin{proof}
\par The definition of a pattern and the second part of the definition of well-distributed is easy to verify by inspecting the figure, so we focus on checking pairwise separability for vertices and colours. Hence, we focus on the first part of the statement for each individual pattern.
\begin{itemize}
    \item $P_{x,z}(y,\ell_1, \ell_2)$. For this pattern and the next, we implicitly assume that $x\neq z$, making $x$ and $z$ separable by \ref{gnotzero}.  All other pairs of vertices are separable by \ref{linear}. Except for the two pairs of colours $(2\ell_1-y-x, y-\ell_1)$ and $(2\ell_2-y-z, z-\ell_2)$ which are separable by \ref{32}, all pairs of colours are separable by \ref{linear}.
     \item $P'_{x,z}(y,\ell_1, \ell_2)$. This pattern is well-distributed as it is a strict subset of $P_{x,z}(y,\ell_1, \ell_2)$.
     \item $P_{x}(c_1,c_2)$. All pairs of vertices and colours are separable by \ref{linear}.
     \item $P_{x}^{(4)}(c_1,c_2,c_3)$. All pairs of vertices and colours are separable by \ref{linear}.
     \item $P_a(v,d_1,d_2)$. All pairs of vertices and colours are separable by \ref{linear}.
\end{itemize}
Hence, each pattern is well-distributed as desired.
\end{proof}
The following remark is crucial to keep in mind as we will make many references to Figure~\ref{fig:gadgets} in the rest of the section.
\begin{remark}\label{rem:labelsarefunctions} On numerous occasions, we will consider multiple \textit{instances} of the same \textit{pattern type} in Figure~\ref{fig:gadgets}. For example, $P_{x,z'}(y',\ell_1,\ell_2)$ refers to the pattern depicted in Figure~\ref{fig:gadgets} denoted $P_{x,z}(y,\ell_1,\ell_2)$ with $y$ replaced with $y'$, $z$ replaced with $z'$, but $\ell_1$ and $\ell_2$ unchanged. So $P_{x,z}(y,\ell_1,\ell_2)$ and $P_{x,z}(y',\ell_1,\ell_2)$ are different patterns with some overlap in the vertex and edge labels they receive.  
\end{remark}

\subsubsection{Vertex-switchers}
We first show how to $1$-absorb a set of $2$ vertices. We emphasise that when we say a vertex of $\mathcal{H}_k$ in this section, we specifically mean a vertex which corresponds to a vertex of $\vec{K}_G$, as opposed to a colour of $\vec{K}_G$. This is crucial, as in fact it is impossible to build a $1$-absorber for a set of $2$ vertices of $\mathcal{H}_k$ which correspond to colours of $\vec{K}_G$. This follows from Observation~\ref{obs:key}.
\begin{lemma}\label{Lemma_absorber_pair_main}
Let $p\geq n^{-1/700}$. Let $3 \leq k\leq \log^{10} n$. Let $R_1, R_2$ be $p$-random subsets of $G$ sampled independently. With high probability, the following holds.

For any distinct vertices $x,z\in \mathcal{H}_k$ and $U\subseteq G$ with $|U|\leq p^{300}n/C_{\ref{Lemma_absorber_pair_main}}$,
$\mathcal{H}_k[(R_1\setminus U)\cup\{x,z\}, (R_2\setminus U)]$ contains a subgraph of size at most $10k$ that $1$-absorbs $\{x,z\}$.
\end{lemma}
\begin{proof} With high probability, Lemma~\ref{lem:patternfinding} and Lemma~\ref{lem:exhaustingvertices} both hold.
\par Suppose first that $k=3$. Consider the pattern $P=P_{x,z}(y,\ell_1, \ell_2)$ from Figure~\ref{fig:gadgets}. A copy of $P$ can be found in $\vec{K}_G[(R_1\setminus U)\cup\{x,z\}, (R_2\setminus U)]$ by Lemma~\ref{lem:patternfinding}. Inspecting the two sets of matchings (the solid matching and the dashed matching) in Figure~\ref{fig:gadgets} (top-left), we see that such a copy of $P$ corresponds to a desired absorbing subgraph of $\mathcal{H}_k[(R_1\setminus U)\cup\{x,z\}, (R_2\setminus U)]$ of size $\leq 30$. Recall Observation~\ref{patternsaresubgraphs} to justify this correspondence.
\par Suppose now that $4\leq k\leq \log^{10} n$. In this case consider a copy of the pattern $P=P'_{x,z}(y,\ell_1, \ell_2)$ from Figure~\ref{fig:gadgets} given by Lemma~\ref{lem:patternfinding}. To complete the absorber, we need to find rainbow paths of size $k-2\geq 2$ from $\ell_1$ to $y-\ell_1+x$ and from $\ell_2$ to $y-\ell_2+x$ (when we say $\ell_1$, we mean the copy of the vertex with the label $\ell_1$ in $P'_{x,z}(y,\ell_1, \ell_2)$, similarly for the other variables. We use this convention in the rest of this section). Two applications of Lemma~\ref{lem:exhaustingvertices}(1) allows us to find these paths. As in the previous case, the solid matching and the dashed matching demonstrates that the desired absorption property holds.
\end{proof}

Chaining together gadgets which $1$-absorb pairs as in the previous lemma, we can construct gadgets which $(s-1)$-absorb sets of size $s$.
\begin{lemma}\label{Lemma_absorber_99_main}
Let $p\geq n^{-1/700}$. Let $3 \leq k\leq \log^{10} n$. Let $R_1, R_2$ be $p$-random subsets of $G$ sampled independently. With high probability, the following holds.

Let $S$ be a vertex-subset of size at most $100$ and let $U\subseteq G$ with $|U|\leq p^{310}n/C_{\ref{Lemma_absorber_99_main}}$. Then,
there are sets $V'\subseteq R_1\setminus U$ and $C'\subseteq R_2\setminus U$ of size $\leq 10^4k$ such that $V'\cup C'$ $(|S|-1)$-absorbs $S$ in $\mathcal{H}_k$.
\end{lemma}
\begin{proof}
With high probability, Lemma~\ref{Lemma_absorber_pair_main} holds.
\par Let $S=\{a_1,\ldots, a_\ell\}$ be given for some $\ell$ with $2\leq \ell\leq 100$. For every $i\in \{1,\ldots, \ell-1\}$, apply Lemma~\ref{Lemma_absorber_pair_main} with $(a,b)=(a_i,a_{i+1})$, each time finding a subgraph $F_i$ which $1$-absorbs $\{a_i,a_{i+1}\}$ disjointly with $U$. By extending $U$ in each application to include $F_1,\ldots, F_{i-1}$, we can also ensure the collection of $F_i$ are vertex and colour disjoint (except for elements of $S$). The union of the subgraphs $F_i$ ($i\in [\ell]$) now has the desired absorption property. For an illustration of the case when $k=\ell=3$, see Figure~\ref{fig:2to3} (left).
\end{proof}

Now we show how to $1$-absorb sets of size $s\leq 100$. For small $k$, essentially all the necessary ideas are included in the previous two lemmas, however for large $k$ we introduce some new ideas.

\begin{figure}
    \centering
    \includegraphics[width=\textwidth]{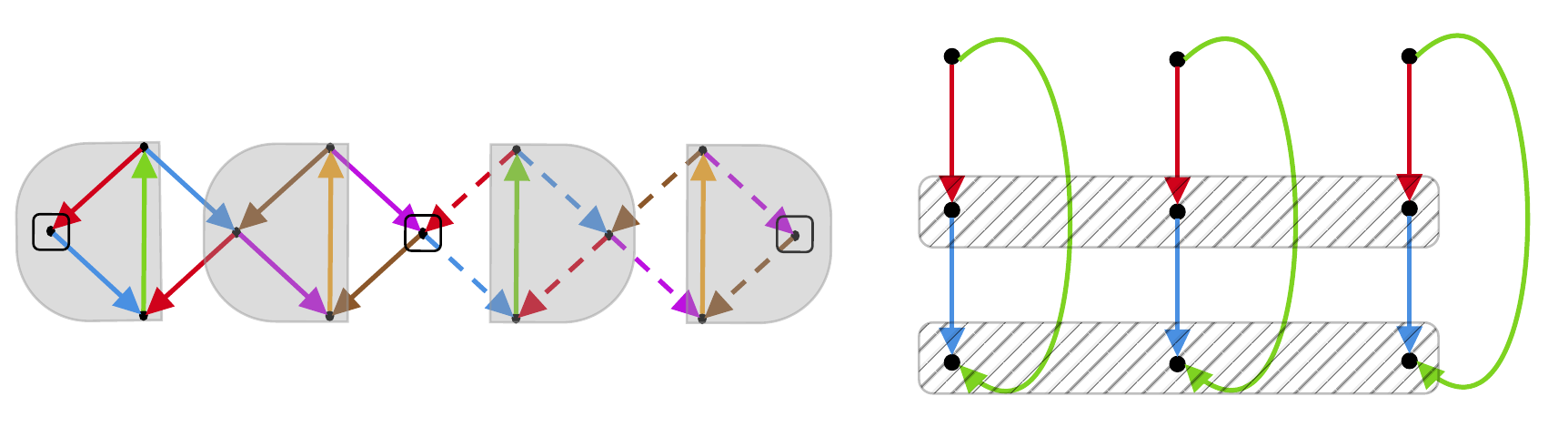}
    \caption{\textit{On the left:} A $2$-absorber for the for the $3$ vertices contained in boxes (from the proof of Lemma~\ref{Lemma_absorber_pair_main}). Dashed versions of a coloured edge are to be interpreted as having a distinct colour. The matching which absorbs the outer two vertices is indicated. \textit{On the right:} The pattern $P_S$ (from the proof of Lemma~\ref{Lemma_vertex_absorber_2to100} Case $1$) consisting of $3$ copies of the pattern $P_x(c_1,c_2)$ ($\ell=3$), and the sets $S_2$ and $S_3$ shaded with the diagonal lines. For the vertices covered with the diagonal lines we have a $2$-absorber for the indicated vertices. The union of these two $2$-absorbers and the illustrated directed graph $1$-absorbs the $3$-vertices on the top row of the diagram.}
    \label{fig:2to3}
\end{figure}

\begin{lemma}\label{Lemma_vertex_absorber_2to100}
Let $p\geq n^{-1/700}$ and let $3 \leq k\leq \log^{10} n$. Let $R_1, R_2$ be $p$-random subsets of $G$ sampled independently. With high probability, the following holds.

Let $S$ be a vertex-subset of size at most $100$ and let $U\subseteq G$ with $|U|\leq p^{340}n/C_{\ref{Lemma_vertex_absorber_2to100}}$,
there are sets $V'\subseteq R_1\setminus U$ and $C'\subseteq R_2\setminus U$ of size $\leq 10^{8}k^2$ such that $V'\sqcup C'$ $1$-absorbs $S$ in $\mathcal{H}_k$.
\end{lemma}
\begin{proof} With high probability, Lemma~\ref{lem:patternfinding}, Lemma~\ref{lem:exhaustingvertices} and Lemma~\ref{Lemma_absorber_99_main} all hold. Denote $S=\{a_1,a_2, \ldots , a_\ell\}$ where $2\leq \ell \leq 100$.
\par \textbf{Case 1:} $k\in \{3,4\}$. We write the details of the argument for $k=3$, for $k=4$ a proof can be obtained by replacing $P_{a_i}(c_1,c_2)$ with $P_{a_i}^{(4)}(c_1,c_2,c_3)$ in the below argument\footnote{The same argument works for each $k=O(1)$, the reason why the second case exists is the range when $k=\Omega(\log n)$.}.
\par Consider a pattern $P_S$ formed by the union of patterns $P_{a_i}(c_1,c_2)$ for each $i\in [\ell]$ (recall Remark~\ref{rem:labelsarefunctions}). This pattern is well-defined as the vertices of $P_{a_i}(c_1,c_2)$ and $P_{a_j}(c_1,c_2)$ get different labels when $i\neq j$ (as $a_i\neq a_j$). To see that this pattern is well-distributed we only need to check vertices and edges belonging to different copies, as each $P_{a_i}(c_1,c_2)$ is well-distributed by Lemma~\ref{lem:everythingwelldistributed} already. There is nothing to check with edge labels as they are all identical. Vertices in the same position in the triangle are separable by \ref{gnotzero} (this is as $a_i\neq a_j$ when $i\neq j$), and vertices in different positions are separable by \ref{linear}. By Lemma~\ref{lem:patternfinding}, $P_S$ admits a copy in $\vec{K}_G[(R_1\setminus U)\cup S;R_2\setminus U]$, say $P_S'$. For $j\in\{2,3\}$, denote by $S_2$ the vertices of $F_S'$ coming from copies of top vertices of $P_{a_i}(c_1,c_2)$ and denote by $S_3$ the vertices coming from bottom vertices of  $P_{a_i}(c_1,c_2)$ for each $i\in[\ell]$. Apply Lemma~\ref{Lemma_absorber_99_main} first with $S=S_2$, and then $S=S_3$, to find sets $A_2$ and $A_3$ $(|S|-1)$-absorbing $S_2$ and $S_3$, respectively. We can also ensure that $A_2$, $A_3$ and $F_S'$ are vertex and colour disjoint by extending $U$ in each successive application of Lemma~\ref{Lemma_absorber_99_main}. Now $F_S'\cup A_2\cup A_3$ is the desired absorber. See Figure~\ref{fig:2to3} for a demonstration of this when $\ell=3$, the boxes shaded with diagonal lines represent the sets $S_2$ and $S_3$ which we can $2$-absorb.

\par \textbf{Case 2:} $k\geq 5$. We begin by the following observation which will help to motivate the choice of parameters for the rest of the argument. We remind the reader that when we say $d_1+d_2-2v-a$ below, what we mean is the copy of the vertex of $P_a(v,d_1,d_2)$ with the label $d_1+d_2-2v-a$, and similarly for the other expressions.
\begin{observation}\label{obs:motivating}
    Consider a copy $P$ of $P_a(v,d_1,d_2)$ from Figure~\ref{fig:gadgets}. Let $Q$ be a rainbow path from $d_1+d_2-2v-a$ to $v+a$, colour-disjoint with $P$, and vertex-disjoint with $P$ except on the endpoints. Let $A$ be a set that $2$-absorbs $Y:=\{d_1-v, d_1-v+d_2-a, d_2-v\}$ (the column shaded by the diagonal lines), vertex and colour disjoint with $P\cup Q$, except on $Y$. Then, $P\cup A\cup Q\setminus \{a,d_1,d_2\}$ $1$-absorbs the vertex-triple $\{a,d_1,d_2\}$.
\end{observation}
 The previous observation essentially shows that we can find sets $1$-absorbing triples, provided that $2$ elements of the triples are \textit{free variables} (it is important that $d_1$ and $d_2$ are free variables while ensuring that $P_a(v,d_1,d_2)$ is well distributed, and that $d_1+d_2-2v-a$ is linear in at least one variable). The remainder of the proof is focused on using this property many times with carefully chosen sets of triples to find subgraphs that $1$-absorb sets of size at most $100$ constants (not free variables). The following observation motivates the choice of triples.

\begin{observation}\label{obs:23to100}
    Consider a collection of sets $\{a_1,d_\ell,d_1\}$, $\{a_2,d_1,d_2\}$, $\{a_3,d_2,d_3\}$, $\ldots$, $\{a_{\ell-1},d_{\ell-2},d_{\ell-1}\}$, $\{a_\ell, d_{\ell-1}\}$. Suppose for some $a_i$, the set containing $a_i$ is deleted from the collection. Then, there exists a choice of an element from each of the remaining sets in the collection so that overall the chosen elements are precisely $\{d_1,\ldots, d_{\ell-1}\}$.
\end{observation}
\begin{proof}
    When $i=3$, the correct choices are displayed in the below table.
\begin{center}
\begin{tabular}{ c c c c c c}
 $a_1$ & $a_2$ & $\mathbf{a_3}$ & $a_4 \,\,\, \cdots $ & $a_{\ell-1}$ & $a_\ell$\\
 $d_\ell$ & $d_1$ & $d_2$ & $\mathbf{d_3}\,\,\, \cdots $ & $\mathbf{d_{\ell-2}}$ & $\mathbf{d_{\ell-1}}$\\
 $\mathbf{d_1}$ & $\mathbf{d_2}$ & $d_3$ & $d_4\,\,\, \cdots $ & $d_{\ell-1}$
\end{tabular}
\end{center}
For other values of $i$, one can similarly choose the value on the bottom row for the $j$th column where $j<i$, and choose the value on the middle row for the $j$th column where $j>i$.
\end{proof}

Consider a pattern $P$ formed by the union of patterns $P_{a_1}:=P_{a_1}(v_1,d_\ell,d_1)$, $P_{a_2}:=P_{a_2}(v_2,d_1,d_2)$, $P_{a_3}:=P_{a_3}(v_3,d_2,d_3)$, $\ldots$, $P_{a_{\ell-1}}:=P_{a_{\ell-1}}(v_{\ell-1},d_{\ell-2},d_{\ell-1})$ (recall Remark~\ref{rem:labelsarefunctions}). Formally, $P$ is obtained by taking the (disjoint) union of each of the graphs $P_{a_i}$ and identifying together vertices which share the same label (this identification step ensures that the second property in the definition of pattern is satisfied, so $P$ is indeed a well-defined pattern). We remark that the vertex/edge labels live in the set $(G\ast F_{2\ell-1})^{\mathrm{ab}}$.
\begin{claim} $P$ is well-distributed.
\end{claim}
\begin{proof}
Each $P_{a_i}(\cdot ,\cdot ,\cdot)$ is well-distributed by Lemma~\ref{lem:everythingwelldistributed}, so we only need to check separability for pairs of vertices and pairs of colours coming from different $P_{a_i}(\cdot ,\cdot ,\cdot)$. For pairs of such colours, the word corresponding to a label includes $v_i$ as a free variable for some $i$, and the other label does not include $v_i$. This makes the pair separable by \ref{linear}. Similarly, for pairs of vertices, if one of the words has a $v_i$ as a free variable for some $i$, we have separability by \ref{linear}. If this is not the case the pair of words could be of the form $(d_i,d_j)$ for some $i$ and $j$. If $i\neq j$, we have separability by \ref{linear}, and otherwise, the words are identical and in this case we do not need to check separability, as in the definition of the union we identify such vertices together. The pair of words could also be of the form $(d_i, d_j+d_{j+1}-2v_j-a_j)$, in which case either $d_j$ or $d_{j+1}$ is different with $d_i$, giving separability by \ref{linear}. Finally, pair of words could also be of the form $(d_i+d_{i+1}-2v_i-a_i, d_j+d_{j+1}-2v_j-a_j)$ where $i$ and $j$ are distinct, so we again have separability by \ref{linear}.
\end{proof}
By Lemma~\ref{lem:patternfinding}, we can find a copy of $P$, say $P'$, in $\vec{K}_G[R_1\setminus U;R_2\setminus U]$. In addition, by Lemma~\ref{Lemma_absorber_pair_main} applied with $\{a,b\}=\{a_\ell, d_{\ell-1}\}$, we obtain a subgraph $A$ which $1$-absorbs $\{a_\ell, d_{\ell-1}\}$. By extending $U$ in this application, we can ensure $A$ is disjoint with $P'$.

\par Recall that $k-3\geq 2$ by assumption. For each $i\in[\ell-1]$, apply Lemma~\ref{lem:exhaustingvertices} with $u$ set to be the copy of the rightmost vertex in $P_{a_i}$ and $v$ set to be the copy of the leftmost vertex in $P_{a_i}$ to find a rainbow path of length $k-3$ directed from $u$ to $v$ that does not clash with any of the forbidden colours or vertices (see the dashed line in Figure~\ref{fig:gadgets}). We can achieve this by iteratively invoking Lemma~\ref{lem:exhaustingvertices}, extending $U$ at each step. Note that $k\leq \log^{10} n$, so we never add more than $200\log^{10} n$ elements to $U$ during this process. Again, for each $i\in[\ell-1]$, apply Lemma~\ref{Lemma_absorber_99_main} with $S$ set the be the subset corresponding to the copies of the $3$ vertices of $P_{a_i}$ which correspond to the column indicated with the dots to obtain a subgraph $A_{a_i}$ which $2$-absorbs this subset. We can ensure that the collection of $A_{a_i}$ are pairwise disjoint (except for the vertices corresponding to highlighted vertices plugged into $S$), and also disjoint with $U$, again by extending $U$ in each application of Lemma~\ref{Lemma_absorber_99_main}. By the bound coming from Lemma~\ref{Lemma_absorber_99_main}, we never have to extend $U$ by more than $10^6k^2\ell$ elements during this process.
 \par For each $i\in[\ell -1]$, let $P'_{a_i}$ be the copy of $P_{a_i}$ combined with the path found by applying Lemma~\ref{lem:exhaustingvertices} and $A_{a_i}$, and let $Z_{a_i}$ be the set of vertices corresponding to the copies of the vertices of $P_{a_i}$ (the column indicated by dots). The following is a rephrasing of Observation~\ref{obs:motivating}.
\begin{observation}\label{obs:1from3}
We have that $P'_{a_i}\setminus Z_{a_i}$ $1$-absorbs $Z_{a_i}$.
\end{observation}
Now, we claim that $(\bigcup P'_{a_i}\setminus S)\cup A$ is the desired absorber. To see this, let $a_i\in S$. We wish to show that the vertices and colours of $(\bigcup P'_{a_i}\setminus S)\cup A \cup \{a_i\}$ induce a perfect matching in $\mathcal{H}$. To see this, first take a perfect matching in $(P_{a_i}'\setminus Z_{a_i})\cup \{a_i\}$ (which exists by definition of $1$-absorbing and Observation~\ref{obs:1from3}).

\par Now, from each set $Z_{a_j}$ (for $i\neq j$) and $\{a_\ell, d_{\ell-1}\}$ (coming from the $1$-absorbing $A$), we can select exactly one $d_{j'}$ so that each $d_{j'}$ (for each $j'\in [\ell-1]$) is selected precisely once, using Observation~\ref{obs:23to100}. Using either Observation~\ref{obs:1from3} or the $1$-absorbing property of $A$, we can find a perfect matching of $(\bigcup P'_{a_i}\setminus S)\cup A\cup \{a_i\}$ as required.
\end{proof}

\subsubsection{Colour-switchers}

\begin{lemma}\label{Lemma_colour_absorber}
Let $p\geq n^{-1/700}$. Let $5 \leq k\leq \log^{10} n$. Let $R_1, R_2$ be $p$-random subsets of $G$ sampled independently. With high probability, the following holds.

\par Let $\alpha\in G$, and $s\in\mathbb{N}$ with $2\leq s\leq \min\{k-2, 100\}$. Let $S$ be a disjoint and near-dissociable family of rainbow $s$-tuples of colours, each tuple sums to $\alpha$, and each tuple is a path-candidate, and suppose $|S|\leq 100$. Let $U\subseteq G$ with $|U|\leq p^{400}n/C_{\ref{Lemma_colour_absorber}}$. Then, there are sets $V'\subseteq R_1\setminus U$ and $C'\subseteq R_2\setminus U$ of size $\leq 10^{11}k^2$ such that $V'\cup C'$ $1$-absorbs $S$ in $\mathcal{H}_k$.
\end{lemma}
\begin{proof}
With high probability, Lemma~\ref{lem:patternfinding}, Lemma~\ref{Lemma_absorber_99_main} and Lemma~\ref{lem:exhaustingvertices} hold. Let $S$ and $U$ be given as in the statement. Consider a pattern $P$ constructed as follows. Take $|S|$ directed paths, each of length $s$, with the same start and end vertices, but internally vertex-disjoint. Label the start vertex with the free variable $v$. Label the edges of the $i$th directed path with the $i$th $s$-tuple of $S$ (counting in order of proximity to $v$). Note this induces a labelling on each of the remaining vertices of $P$. This labelling is well-defined on the end-vertex, because each tuple in $S$ has the same sum, namely, $\alpha$. In particular, the end-vertex receives the label $v-\alpha$. For an illustration of the pattern when $s=3$ and $S=\{(e_1,e_2,e_3), (d_1,d_2,d_3), (c_1,c_2,c_3)\}$, inspect the bottom-right pattern in Figure~\ref{fig:gadgets}.
\begin{observation}
$P$ is well-distributed.
\end{observation}
\begin{proof}
Each pair of colours is separable by \ref{gnotzero} as distinct coordinates of elements of $S$ are distinct, and the elements of $S$ are pairwise disjoint. We claim each pair of vertices is separable by \ref{gnotzero}. For vertices belonging to the same directed path, this follows as elements of $S$ are path-candidates. For vertices belonging to different directed paths, this follows as $S$ is near-dissociable.
\end{proof}
Thus, we may apply Lemma~\ref{lem:patternfinding} to find a copy of $P$, say $P'$, in $\vec{K}_G[R_1\setminus U; (R_2\setminus U)\cup \bigcup S]$. For each $i\in[s-1]$, let $P_i$ denote the vertices of distance $i$ from $v$, noting $|P_i|=|S|\leq 100$. Apply Lemma~\ref{Lemma_absorber_99_main} for each $P_i$ to find (disjointly) sets $P_i'$ which $|S|-1$ absorb $P_i$. Finally, noting that $k-s\geq 2$, apply Lemma~\ref{lem:exhaustingvertices} with $u$ as the end-vertex of the path of length $s$ from $v$, and $v=v$, and $k'=k-s$, to find a rainbow path of length $k-s$. It is easy so see that the resulting structure has the desired absorption property.
\end{proof}

\subsubsection{Putting the gadgets together}
So far we have lemmas allowing us to find sets $1$-absorbing arbitrary sets of size $\leq 100$. How can we go from to sets which can $h$ absorb sets of size $(1+\beta) h$, where $h$ and $\beta h$ are potentially linear in $n$? The bipartite graph given in the next lemma gives us a nice collection of subsets of size at most $100$ to $1$-absorb, which together have the desired property. In this sense the utility of this bipartite graph is similar in spirit to that of the sequence constructed in Observation~\ref{obs:23to100} (which could be viewed as a bipartite graph with maximum degree $3$ with much weaker properties).
\begin{lemma}[Montgomery, \cite{randomspanningtree}]\label{lem:robustbipartite} Let $0<\beta\leq 1$. There is a positive integer $h_0$ such that for every $h\geq h_0$ there exists a bipartite graph $K$ with maximum degree at most $100$ and vertex classes $X$ and $Y\cup Y'$ with $|X|=3h$, $|Y|=2h$, $|Y'|=h+\beta h$ so that the following holds. For any $Y_0\subseteq Y'$ with $|Y_0|=h$, there is a perfect matching between $X$ and $Y\cup Y'$.
\end{lemma}
Graphs produced by this lemma are called \emph{robustly matchable bipartite graphs}.

\begin{lemma}\label{lemma:flexiblevertexcolourabsorber}
Let $p\geq n^{-1/700}$, $3 \leq k\leq \log^{10} n$. Let $R_1, R_2$ be $p$-random subsets of $G$ sampled independently. With high probability, the following holds.
\par Let $0<\beta\leq 1$. Let $ h\in \mathbb{N}$ with $\sqrt{n}\leq h\leq p^{400}n/C_{\ref{lemma:flexiblevertexcolourabsorber}}k^2$. Let $U\subseteq G$ with $|U|\leq n^{999/1000}$. Let $Y'$ be a subset of size $(1+\beta)h$ with one of the following forms.
\begin{enumerate}
    \item $Y'\subseteq \mathcal{H}_k$ is a vertex-subset of $\mathcal{H}_k$.
    \item Let $\alpha\in G$, and $s\in\mathbb{N}$ with $2\leq s\leq  \min\{k-2, 100\}$. $Y'$ is a disjoint and near-dissociable family of $s$-tuples of colours of $\mathcal{H}_k$ where each tuple sums to $\alpha$, and each tuple is a path-candidate (hence $\alpha\neq 0$).
\end{enumerate}
Then, there exists a set $A\subseteq \mathcal{H}_k\setminus U$ of size $\leq 10^{15}k^2 3h$ where $A$ $h$-absorbs $Y'$.
\end{lemma}
\begin{proof}
    We show how to prove part (1) of the statement using Lemma~\ref{Lemma_vertex_absorber_2to100}. The proof for part (2) is essentially the same, using Lemma~\ref{Lemma_colour_absorber} instead.
    \par With high probability, Lemma~\ref{Lemma_vertex_absorber_2to100} holds and $R_2$ has size at least $10h$.
    \par Let $\beta$, $h$, $U$, and $Y'$ be given. As $h$ is sufficiently large, we can apply Lemma~\ref{lem:robustbipartite} to construct a bipartite graph $G$ with parameters $\beta$ and $h$ with vertex classes $X$ and $Y\cup Y'$. Here, we associate the given set of vertices $Y'$ with the $Y'$ that denotes a set of vertices of $G$. We arbitrarily associate, disjointly with $Y'$ and $U$, a subset of vertices of $\mathcal{H}_k$ from $R_2$ with $Y$ ($|R_2|\geq 10h$, so there is space to do this). Now, each element of $x\in X$ is linked, via the graph $G$, to a subset of vertices of $\mathcal{H}_k$ of size at most $100$, i.e. the neighbourhood which we denote $N_G(x)$. For each element of $x\in X$, we will find a $A_x\subseteq \mathcal{H}_k$ that $1$-absorbs $N_G(x)$, and the collection of $A_x$ we find will be disjoint except on elements of $Y\cup Y'$. The property from Lemma~\ref{Lemma_vertex_absorber_2to100} allows us to do this greedily, extending $U$ with $10^{8}k^2$ elements at each step, adding to $U$ at most $$10^8k^2|X|\leq 10^8k^2 3h\leq  3 \cdot 10^8k^2 p^{400}n/(C_{\ref{lemma:flexiblevertexcolourabsorber}}k^2 )\leq 3\cdot 10^8p^{400}n/C_{\ref{lemma:flexiblevertexcolourabsorber}}$$
    elements. Combined with the initial elements of $U$, this means that $U$ never exceeds a size of $p^{340}n/C_{\ref{Lemma_vertex_absorber_2to100}}$ if $C_{\ref{lemma:flexiblevertexcolourabsorber}}$ is sufficiently large. This means that the applications of Lemma~\ref{Lemma_vertex_absorber_2to100} are valid.
    \par We claim that the union of the $A_x$ with $Y$ have the desired absorption property. To see this, take some subset $Y_0\subseteq Y'$ of size $h$. In $G$, we have perfect matching $f$ from $X$ to $Y\cup Y_0$. For each $A_x$, use the matching of $A_x\cup \{f(x)\}$ which exists by the absorption property of $A_x$. These matchings together give a matching of $\bigcup A_x \cup Y \cup Y_0$, as required.

    \end{proof}

\subsection{Proof of Lemma~\ref{lem:zerosumabsorption}}
Now we combine the distributive absorption strategy with the cover-down strategy to give a proof of Lemma~\ref{lem:zerosumabsorption}. Recall the convention about random subsets of random sets given before the proof of Theorem~\ref{thm:mainthm}.
\begin{proof}
\par Let $K=K_{\ref{lem:zerosumabsorption}}\geq 1$ be sufficiently large, and fix $\eps=\eps_{\ref{lem:zerosumabsorption}}\ll 1/K$, so that in particular, $\eps K\leq 10^{-10}$ holds.
\par \textbf{Case 1:} $3 \leq k \leq 9$. Write $p=p_1+p_2$ where $p_2=p^{500}/(10C_{\ref{lemma:flexiblevertexcolourabsorber}}k^3)$. Partition $R_i$ into disjoint $p_1$ and $p_2$-random sets $R_i^{(1)}$ and $R_i^{(2)}$ for each $i\in [2]$. Let $R_i^*\subseteq R_i^{(2)}$ be a $r$-random subset of $G$ where $r=p^{500}/(1000C_{\ref{lemma:flexiblevertexcolourabsorber}}k^3)$.  With high probability, $R_1^{(2)}$ satisfies Lemma~\ref{lem:exhaustcolourssmallk}, $R_2^*$ satisfies Lemma~\ref{lem:generalisedtannenbaum} as well as Lemma~\ref{Lemma_find_set_with_correct_sum}, $(R_1^*,R_2^*)$ satisfies Lemma~\ref{lem:exhaustingvertices}, and $(R_1^{(1)},R_2^{(1)})$ satisfies Lemma~\ref{lemma:flexiblevertexcolourabsorber} (the necessary lower bounds for the corresponding randomness parameters in each of these applications is satisfied for a small enough value of $\eps_{\ref{lem:zerosumabsorption}}$). With high probability, the size of each random set is at most $n^{0.6}\log n$ away from its expectation. All these properties hold simultaneously with high probability.
\par Now, let $U\subseteq G$ with $|U|\leq n^{4/5}$, without relabelling, include $0$ in $U$. By Lemma~\ref{Lemma_find_set_with_correct_sum}, we can find a subset $R_2^{**}\subseteq R_2^{*}\setminus U$ using all but at most $k$ elements of $R_2^{*}\setminus U$ such that $\sum R_2^{**}=0$ and $k$ divides $|R_2^{**}|$. Set $\beta$ and $h$ so that they satisfy the two identities $(1+\beta)h=|R_1^{(2)}\setminus U|$ and $\beta h = |R_2^{**}|$ (so $h=|R_1^{(2)}\setminus U|-|R_2^{**}|\geq \sqrt{n}$, and $0<\beta\leq 1$ by choice of $r$). Apply Lemma~\ref{lemma:flexiblevertexcolourabsorber}(1) with these values of $\beta$ and $h$ and $Y':=R_1^{(2)}\setminus U$ to obtain an absorbing set $A$ contained in $R_1^{(1)}\cup R_2^{(1)}\setminus U$ (the necessary upper bound on $h$ holds by definition of $p_1,p_2$ using that each random set has size close to its expectation). 

We claim now that $A\cup R_2^{**}\cup (R_1^{(2)}\setminus U)$ has the desired property (of $V\cup C$ in the statement). To see this, take $V',C'\subseteq G$ with $|V'|=|C'|=m$ as in the statement of the lemma. As $m\ll r^{300}n/kC_{\ref{lem:exhaustingvertices}}$ (supposing $K_{\ref{lem:zerosumabsorption}}$ is sufficiently large), by Lemma~\ref{lem:exhaustingvertices}(2), there exists a matching $M_1$ of size exactly $m$ in $\mathcal{H}_k$ saturating $V'$ and using exactly $(k-1)m$ vertices from $R_1^*\setminus U$ and $km$ vertices from $R_2^{**}$. $C'':=C'\cup (R_2^{**}\setminus V(M_1))$ is a zero-sum set whose order is divisible by $k$ with small symmetric difference with $R_2^{*}$ (note that $|C''|=m+|R_2^{**}|-km$, so $|C''\Delta R_2^{*}|\leq 10km + n^{4/5} \leq 10k(p/k\log n)^K n + n^{4/5}\leq r^{10^{10}}n/\log(n)^{10^{23}}$ supposing $K$ is sufficiently large). Hence $C''$ can be partitioned into $k$-sets which are cycle-candidates by Lemma~\ref{lem:generalisedtannenbaum}(1). This partition allows us to apply Lemma~\ref{lem:exhaustcolourssmallk} to deduce that there exists a matching $M_2$ saturating the colours $C''$ using exactly $|C''|$ vertices from the set $R_1^{(2)}\setminus U\setminus V(M_1)$. Observe that in total we used exactly $|R_2^{**}|=\beta h$ vertices from $R_1^{(2)}\setminus U$, and therefore the remaining vertices in $R_1^{(2)}\setminus U$ combined with $A$ admits a perfect matching $M_3$ by the absorption property of $A$. Then, $M_1\cup M_2\cup M_3$ is the desired perfect matching of $A\cup R_2^{**}\cup V'\cup C'\cup (R_1^{(2)}\setminus U)$.

\par \textbf{Case 2:} $k\geq 10$. Set $q_1=p^{500}/(10^{10}C_{\ref{lemma:flexiblevertexcolourabsorber}}k^2)$, $q_2=(p-q_1)/3$, $r_{*}=q_1/1000k^{10}$. Let $R_1^{(1)}$, $R_1^{(2)}$, $R_1^{(3)}$, $R_1^{(4)}$ be disjoint subsets of $R_1$, and $q_1$, $q_2$, $q_2$, $q_2$-random, respectively. Let $R_1^{(1,1)},R_1^{(1,2)} \subseteq R_1^{(1)}$ be $r_*$-random and $(q_1-r_*)$-random and disjoint. Let $R_2^{*}$, $R_2^{(1)}$, $R_2^{(2)}$, $R_2^{(3)}$, $R_2^{(4)}$ be disjoint subsets of $R_2$ and $r_*$, $(q_1-r_*)$, $q_2$, $q_2$ and $q_2$-random, respectively.

With high probability, Lemma~\ref{lem:exhaustingvertices} holds for ($R_1^{(1,1)}$, $R_2^*$), Lemma~\ref{lem:generalisedtannenbaum} holds for $R_2^*$, Lemma~\ref{lem:exhaustingcolourslargek} holds for ($R_1^{(1,2)}$, $R_2^{(1)}$), Lemma~\ref{lemma:flexiblevertexcolourabsorber} holds for each of $(R_1^{(2)}, R_2^{(2)})$, $(R_1^{(3)}, R_2^{(3)})$, $(R_1^{(4)}, R_2^{(4)})$, Lemma~\ref{Lemma_find_set_with_correct_sum} holds for $R_2^*$, and the size of each random set is at most $n^{0.6}\log n$ away from its expectation. These applications are valid supposing $\eps_{\ref{lem:zerosumabsorption}}$ is small enough, i.e. $p$ is large enough.

\par Let $U$ be given, as before, include $0$ in $U$. Fix $f$ to be the largest integer bounded above by $|R_2^{*}\setminus U|$ with the property that $f-(k-1)m$ is divisible by $4$. By Lemma~\ref{Lemma_find_set_with_correct_sum}, we can fix a $f$-subset $R_2^{**}\subseteq R_2^{*}\setminus U$ using all but at most $4$ vertices from the latter set such that $\sum R_2^{**} = ((f-(k-1)m)/4)\cdot q_{G,k}$.
\par Set $\beta_1$ and $h_1$ be so that $(1+\beta_1)h_1=|R_1^{(1)}\setminus U|$ and $\beta_1 h_1 = (k-1)m + k(f-(k-1)m)/4$. Denote by $\mathcal{F}_G'$ the family of sets from $\mathcal{F}_G$ which are entirely contained in $R_2^{(1)}\setminus U$. Similarly, denote by $\mathcal{S}_G'$ the family of sets from $\mathcal{S}_G$ which are entirely contained in $R_2^{(1)}\setminus U$. Set $\beta_2$ and $h_2$ so that $(1+\beta_2)h_2=|\mathcal{F}_G'|$ and $\beta_2h_2 = ((f-(k-1)m)/4)(k-4-z_\mathcal{S})/4$ (recall this is an integer by Lemma~\ref{lem:biggroupspecialfamily}). Set $\beta_3$ and $h_3$ so that $(1+\beta_3)h_3=|\mathcal{S}_G'|$ and $\beta_3h_3=(f-(k-1)m)/4$.
\par Apply Lemma~\ref{lemma:flexiblevertexcolourabsorber}(1) with $(R_1^{(2)}, R_2^{(2)})$ and $Y'=R_1^{(1)}\setminus U$ with parameters $(\beta_1,h_1)$ to obtain a set $A_1$ (disjoint with $U$) with a vertex-absorption property. Apply Lemma~\ref{lemma:flexiblevertexcolourabsorber}(2) with $(R_1^{(3)}, R_2^{(3)})$ and $Y':=\mathcal{F}_G'$ with parameters $(\beta_2,h_2)$ to obtain a set $A_2$ (disjoint with $U$ and $A_1$) with a colour-absorption property. Similarly, apply Lemma~\ref{lemma:flexiblevertexcolourabsorber}(2) with $(R_1^{(4)}, R_2^{(4)})$ and $Y':=\mathcal{S}_G'$ with parameters $(\beta_3,h_3)$ to obtain a set $A_3$ (disjoint with $U$, $A_1$, and $A_2$) with a colour-absorption property. For the last two applications, we use that $\mathcal{F}_G$ and $\mathcal{S}_G$ are near-dissociable, contain only path-candidates, and that $k-2\geq z_{\mathcal{S}}, 4$ as $k\geq 10$. These properties come from Lemma~\ref{lem:biggroupspecialfamily}. For all three applications, the necessary upper bound on $h$ holds by definition of $q_1,q_2$ using that each random set has size close to its expectation. The lower bounds on $h$ and that $0< \beta \leq 1$ for the latter two applications follow from lower bounds on the sizes of $\mathcal{F}_G'$ and $\mathcal{S}_G'$ which can be derived from Lemma~\ref{lem:exhaustingcolourslargek} (this is done implicitly in the rest of the argument).

\par We claim that $A_1\cup A_2\cup A_3\cup R_2^{**}\cup (R_1^{(1)}\setminus U)\cup \bigcup \mathcal{F}_G'\cup \bigcup \mathcal{S}_G'$ has the desired absorption property. To see this, let $V'$ and $C'$ be given as in the lemma. By Lemma~\ref{lem:exhaustingvertices}(2), there exists a matching $M_1$ in $\mathcal{H}_k$ saturating $V'$ and using exactly $(k-1)m$ vertices from $R_1^{(1,1)}\setminus U$ and $km$ vertices from $R_2^{**}$. $C'\cup (R_2^{**}\setminus V(M_1)):=C''$ then has size $m+f-km=f-(k-1)m$ which is divisible by $4$ by choice of the integer $f$. Furthermore, $\sum C''=(|C''|/4)\cdot q_{G,k}$ by the sum property on the set $R_2^{**}$. Hence, $C''$ can be partitioned into $4$-tuples with sum $q_{G,k}$ (recall this is not $0$) which are path-candidates by Lemma~\ref{lem:generalisedtannenbaum}(2) (as in the previous case, to check that $C''$ has small symmetric difference with $R_2^{**}$, recall that $K$ is sufficiently large). This partition of $C''$ allows us to apply Lemma~\ref{lem:exhaustingcolourslargek} (with $\ell=(f-(k-1)m)/4$) to deduce that there exists a matching $M_2$ saturating $C''$ using (exactly $k\ell= k(f-(k-1)m)/4$ many) vertices from $R_1^{(1,2)}\setminus U$ and colours from $R_2^{(1)}\setminus U$ which are closed under the families $\mathcal{F}_G$ and $\mathcal{S}_G$, and hence also closed under the families $\mathcal{F}_G'$ and $\mathcal{S}_G'$ (as the colours come from the set $R_2^{(1)}$). Lemma~\ref{lem:exhaustingcolourslargek} also guarantees that $M_2$ uses $\ell(k-4-z_{\mathcal{S}})/4$ elements of $\mathcal{F}_G'$ and $\ell$ elements of $\mathcal{S}_G'$. Thus, there are exactly $h_1$ elements of $R_1^{(1)}\setminus U$, $h_2$ elements of $\mathcal{F}_G'$, and $h_3$ elements of $\mathcal{S}_G'$ that are unused by $M_1\cup M_2$, so the leftovers of these sets combine with $A_1$, $A_2$ and $A_3$ (respectively) to produce perfect matchings, say $M_3$, $M_4$ and $M_5$. Then, $\bigcup_{i\in[5]}M_i$ is the desired matching.
\end{proof}

\section{The high-girth case}\label{sec:highgirth}
In this section, we show how the high girth case of the FGT conjecture follows by results from \cite{muyesser2022random}.

\begin{lemma}[\cite{muyesser2022random}]\label{lem:pathlikemain}
    Let $1/n\ll p\leq 1$, let $t$ be a positive integer between $\log^7(n)$ and $\log^8(n)$, and let $q$ satisfy  $p=(t-1)q$. Let $G$ be an abelian group of order $n$. Let $V_{str}, V_{mid}, V_{end}$ be disjoint random subsets with $V_{str}, V_{end}$ $q$-random and $V_{mid}$ $p$-random. Let $C$ be a $(q+p)$-random subset, sampled independently with the previous sets. Then, with high probability, the following holds.
    \par Let $V_{str}'$, $V_{end}'$, $V_{mid}'$ be disjoint subsets of $G$, let $C'$ be a subset of $G$, and let $\ell=|V'_{mid}|/(t-1)$. Suppose all of the following hold.
    \begin{enumerate}
        \item For each random set $R\in\{V_{str}, V_{mid}, V_{end}, C\}$, we have that $|R\Delta R'|\leq n^{0.6}$.
        \item $\sum V_{str}'-\sum V_{end}'=\sum C'$
        \item $\id \notin C'$ if $G$ is an elementary abelian $2$-group.
        \item $\ell:=|V_{str}'|=|V_{end}'|=|V'_{mid}|/(t-1)=|C'|/t$
    \end{enumerate}
    Then, given any bijection $f\colon V_{str}'\to V_{end}'$, we have that $\vec{K}_G[V_{str}'\cup V_{end}'\cup V_{mid}';C']$ has a rainbow $\vec{P}_t$-factor where each path starts on some $v\in V_{str}'$ and ends on $f(v)\in V_{end}'$.
\end{lemma}

\begin{theorem}\label{thm:mainthmhighgirth}
Let $G$ be an abelian group of order $n$, where $n$ is sufficiently large. Suppose $k$ is some integer such that $k\geq \log ^{9} n$, and $k$ divides $n-1$. Suppose $\sum G = 0$. Then, $\mathcal{H}_k[G\setminus\{0\}; C\setminus\{0\}]$ has a perfect matching.
\end{theorem}
\begin{proof}
    If $k\leq n^{1/10^{{10}^{10}}}$, set $s=k$, otherwise set $s= \lceil \log^{10} n \rceil$.
    \par Partition the group $G$ into disjoint sets twice, independently, as $V_1,\ldots, V_s$ and $C_0,\ldots, C_{s-1}$ where each set is $(1/s)$-random, noting $1/s \geq n^{-1/10^{10^{10}}}$ in either case for $n$ large. Set $t:=\lceil \log^7 n \rceil$.
    \par Lemma~\ref{lem:pathlikemain} holds with high probability with $t$, $V_{mid}=\bigcup_{1\leq i\leq t-2} V_i$ and $C=\bigcup_{0\leq i\leq t-2} V_i$. Lemma~\ref{lem:directedwithdeterministic} holds with random sets $(V_i, V_{i+1}, C_i)$ for each $i$ (where indices are viewed in a cyclic order) and each integer value of $\ell=n/s \pm n^{1-1/10^{10}}$ (we achieve this via a union bound over many applications of Lemma~\ref{lem:directedwithdeterministic}). Also with high probability, all random sets are within $n^{0.6}$ elements of their expectations via Chernoff's bound. By the probabilistic method, fix the random sets so they have all the aforementioned properties.
    \par Suppose first that $k\leq n^{1/10^{{10}^{10}}}$, so $s=k$. By the divisibility assumption and the property coming from Chernoff's bound, we can move $O(n^{0.7})$ elements between the sets $V_i$ without relabelling so that each set $V_i$ has size exactly $(n-1)/k$. Similarly, moving around at most $O(n^{1-10^{6}})$ elements, we can make sure each $C_i$ where $i\geq t-1$ has $(n-1)/k+ \lfloor n^{1-10^5} \rfloor$ elements. Now, apply Lemma~\ref{lem:directedwithdeterministic} (with $\ell=(n-1)/k$) with the triples $$(V_{t-1}, V_t, C_{t-1}), (V_{t}, V_{t+1}, C_{t}), \ldots , (V_{k-1}, V_k, C_{k-1})$$
    to find rainbow matchings saturating the corresponding vertex sets (and missing $\lfloor n^{1-10^5} \rfloor$ colours from each $C_i$, $i\geq t-1$). Note that the union of the matchings found give a rainbow $\vec{P}_{k-t}$ factor where each path is directed from $V_{t-1}$ to $V_k$. Now, we apply Lemma~\ref{lem:pathlikemain} with $V_{str}=V_k$ and $V_{end}=V_{t-1}$, and $C'$ set to be the union of $C$ and the $(k-t+1)\lfloor n^{1-10^5} \rfloor$ unused colours in each $C_i$, $i\geq t-1$. $V_{mid}$ remains unchanged. All but the second hypothesis of Lemma~\ref{lem:pathlikemain} follow easily from our choice of sets. To see that $\sum V_{k} - \sum V_{t-1}= \sum C'$, first note that $\sum V_{t-1}-\sum V_{k}= \sum C''$ where $C''$ is all of the colours used via applications of Lemma~\ref{lem:directedwithdeterministic} (this comes from the fact that we have a rainbow directed path factor in $\vec{K}_G$ where each path starts in $V_{t-1}$ and ends in $V_k$).  As $\sum G\setminus \{0\}=0$ by assumption, and $C'=G\setminus\{0\}\setminus C''$, the desired equality follows. Thus we can indeed apply Lemma~\ref{lem:pathlikemain}. In our application, we set $f$ to be the bijection that maps the last endpoint of each $\vec{P}_{k-t}$ to the first endpoint of the same directed path. This allows us to complete each $\vec{P}_{k-t}$ into a cycle of length $k$, giving us a cycle-factor that corresponds to the desired matching in $\mathcal{H}_k$.
    \par Suppose now that $k>n^{1/10^{{10}^{10}}}$, so $s=\lceil \log^{10} n \rceil$. If it was the case that $s$ divides $k$, then we can proceed exactly like the previous case, with the only difference being in the choice of $f$ in the previous paragraph (we would choose $f$ so that when the connecting paths are found we end up with a $C_k$-factor as opposed to a $C_s$-factor). So suppose that $r$, the remainder when $k$ is divided by $s$, is positive, noting that $r<s$. We start by finding $(n-1)/k$ vertex/colour disjoint rainbow $\vec{P}_r$ in $\vec{K}_G$, calling this collection of paths $\mathcal{P}$. Note this can be done greedily, and the resulting collection of paths occupies $2n^{1-1/10^{10^{10}}}$ vertices, due to our assumption on $k$ and $s$. Let $P_1$ and $P_2$ denote collection of first endpoints of each of the paths in $\mathcal{P}$, respectively. We remove the vertices in $\mathcal{P}\setminus P_2$ from the graph, and proceed exactly as in the previous case to redistribute the sets so that they are of the right size, with the additional condition that $P_2\subseteq V_{t-1}$. In the end, while applying Lemma~\ref{lem:pathlikemain}, we set $V_{end}$ to be $(V_{t-1}\setminus P_2)\cup P_1$. We can then select an appropriate bijection $f$ so that after an application of Lemma~\ref{lem:pathlikemain}, the resulting structure is a $\vec{C}_k$-factor.
\end{proof}

\section{Concluding remarks}\label{sec:concluding}
\subsection{Non-abelian groups, Latin squares, and Ryser's conjecture}
Now that the Friedlander-Gordon-Tannenbaum conjecture is verified, at least for sufficiently large groups, we propose the following extension for general groups.
\begin{conjecture}
    Let $G$ be a sufficiently large group satisfying the Hall-Paige condition, and suppose $k\geq 3$ and $k$ divides $n-1$. Then, there exists an orthomorphism of $G$ that fixes the identity element, and permutes the remaining elements of disjoint cycles of length $k$.
\end{conjecture}
It is also sensible to replace orthomorphisms with complete mappings in the above conjecture, due to the assumption that $k\geq 3$ (recall Remark~\ref{completemapping} from the Introduction). One way of attacking the above conjecture would be to try to combine the methods from this paper with the methods developed for non-abelian groups in \cite{muyesser2022random}. Also, we remark that in Section~\ref{sec:highgirth}, we did not actually use that the group $G$ is abelian. Therefore, the above conjecture is true in the high-girth case.
\par More generally, we can turn our attention to Latin squares, which are also known as quasi-groups. These objects can be described as $n$ by $n$ arrays filled with $n$ symbols so that no symbol repeats in a row or a column. For us, it will be more natural to view Latin squares in the following way (see the survey by Pokrovskiy from \cite{nixon2022surveys} for a more detailed discussion). We first take a complete directed graph $\vec{K}_n$ with edges in both directions between all vertices and a loop at every vertex. We then equip this graph with a proper edge-colouring using $n$ colours. The most famous conjecture in the area is the following.
\begin{conjecture}[Ryser's conjecture]
    Suppose $n$ is odd. Then, $\vec{K}_n$ contains a rainbow spanning subgraph where every vertex has in-degree and out-degree equal to one. Equivalently, $\vec{K}_n$ can be packed with directed cycles in a rainbow fashion.
\end{conjecture}
In analogy with the Friedlander-Gordon-Tannenbaum conjecture, it makes sense to strengthen Ryser's conjecture to ask for cycles of specific lengths. There are numerous conjectures in this area which focus on finding a single cycle which covers the entirety of the vertex set, which is analogous to the $k=n-1$ case of the Friedlander-Gordon-Tannenbaum conjecture. For more information about these conjectures, we refer the reader to Pokrovskiy's survey about rainbow subgraphs in \cite{nixon2022surveys} and a recent paper by Gould and Kelly which includes a nice unifying conjecture \cite{gould2021hamilton}. We pose a conjecture in the other extreme, where the cycle lengths are as small as possible. This is analogous to the $k=3$ case of the Friedlander-Gordon-Tannenbaum conjecture.
\begin{conjecture}\label{mainconjecture}
    Let $K_n$ be a complete graph properly coloured with $n$ colours. Then, $K_n$ contains a rainbow subgraph which is a disjoint union of triangles covering all but at most $C$ vertices, for some absolute constant $C$.
\end{conjecture}
We are not aware of any examples that would rule out the possibility that one can take $C=2$ above. In the other direction, one can prove a relaxed version of the above conjecture with $C$ replaced with $n^{1-\eps}$ for some $\eps>0$ by using the R\"odl nibble (see for example Corollary~\ref{cor:nibble}(1)). Improving this bound, for example by replacing $C$ with a polylogarithmic term, could be an interesting challenge, see \cite{keevash2020new} for an analogous result in the setting of Ryser's conjecture.

\subsection{Other cycle types}
To prove the FGT conjecture, we only used the $p=1$ case of Theorem~\ref{thm:mainthm}. Applying Theorem~\ref{thm:mainthm} with different values of $p$, we can derive that many other cycle types for orthomorphisms are possible. Suppose for example that $G$ is an abelian group of order $n$ with the Hall-Paige property, $n-1=3k+4\ell$, and we want to find an orthomorphism fixing the identity and permuting the remaining elements as $k$ many disjoint $3$-cycles and $\ell$ many $4$-cycles. Let's also suppose for simplicity that $k,\ell=\Omega(n)$. Then, we can partition the vertices of $\vec{K}_G$ into a $3k/n$-random set $V_1$ and $4\ell/n$-random set $V_2$, and we can partition the colours of $\vec{K}_G$ into $3k/n$-random set $C_1$ and $4\ell/n$-random set $C_2$. With positive probability, Theorem~\ref{thm:mainthm}  holds with $(V_1,C_1)$, $k=3$ and with $(V_2,C_2)$, $k=4$. We can then do a few exchanges between the sets of vertices and colours so that they satisfy the divisibility condition as well as the sum condition $\sum C_1=\sum C_2=0$. Then, by Theorem~\ref{thm:mainthm} we obtain the desired cycle partition.
\par We can go further and ask the following question. Suppose that $s_1,s_2,\ldots, s_j $ is a sequence of integers where $s_i\geq 2$ and $\sum s_i=n-1$, and suppose that $G$ is an abelian group with the Hall-Paige property. When is it true that $G$ has an orthomorphism fixing the identity and permuting the remaining elements as cycles of lengths $s_1,s_2,\ldots, s_j$? Note that a necessary condition for the existence of such an orthomorphism is a partition of $G\setminus\{0\}$ into zero-sum sets of size $s_1,s_2,\ldots, s_j$ (recall Observation~\ref{obs:key}). Characterising pairs of sequences $s_1,s_2,\ldots, s_j$ and abelian groups that admit such a partition is known as Tannenbaum's problem. This problem was solved for large groups in \cite{muyesser2022random}. Perhaps the methods from the current paper could be sufficient to solve the more general problem of characterising which cycle types are feasible for orthomorphisms.
\subsection{Other equations}\label{sec:otherequations}
As discussed in Section~\ref{sec:overview}, there is a connection between the Hall-Paige conjecture, the FGT conjecture, and toroidal version of the $n$-queens problem \cite{bowtell2021n}. We can make this connection more formal as follows. Suppose $A$ is a $\ell \times m$ matrix with integer entries, and $G$ is an abelian group of order $n$. Can we find a collection of $n$-many vectors $\vec{v}$ in $G^m$ with $A\cdot \vec{v}=\vec{0}_\ell$ (meaning the $\ell$-dimensional $0$-vector) such that for each $i\in \{1,2,\ldots, m\}$, the collection of $i$th coordinates of the vectors $\vec{v}$ is equal to $G$ (i.e. contains no repetitions). If this is possible, let us call the pair $(A,G)$ \textbf{matchable}. This term is motivated by the fact that we can equivalently phrase this as a hypergraph matching problem in $m$-partite $m$-uniform hypergraphs where the edge set is governed by a collection of $\ell$ linear equations given by the matrix $A$.
\par For example, in the Hall-Paige conjecture, the corresponding matrix $A$ is $[[1,-1,-1]]$, in the $k=3$ case of the FGT conjecture, the matrix is $[[1,-1,0,-1,0,0],[0,1,-1,0,-1,0],[-1,0,1,0,0,-1]]$, and in the $n$-queens problem, the matrix is $[[1,1,-1,0],[1,-1,0,-1]]$. Characterising integer matrices $A$ and abelian groups $G$ such that $(A,G)$ is matchable is a natural unifying problem. This would be interesting already when $A$ consists only of $\{-1,0,1\}$-entries.
\subsection{Controlling the cycle type of both bijections}
It is also natural to investigate the existence of orthomorphisms/complete mappings $\phi$ where one makes a restriction on the cycle type of $\phi$ as well as the cycle type of the permutation $g\to g^{-1}\phi(g)$. Several partial results as well as open problems in this direction are given in \cite{bors2022coset, bors2022cycle} by Bors and Wang. It would be interesting to see if our methods can be adapted to address this more restrictive variant of the problem.
\section*{Acknowledgements} The author thanks Alexey Pokrovskiy for providing feedback on an early version of this manuscript.
\bibliographystyle{abbrv}
\bibliography{bib}
\end{document}